\documentclass[a4paper,12pt,oneside, reqno]{amsart}

\usepackage{amsaddr}
\usepackage[bb=boondox]{mathalfa}
\usepackage{cite}
\usepackage{hyperref}
\usepackage{lmodern} 
\usepackage{amssymb} 

\numberwithin{equation}{section}

\allowdisplaybreaks[4]

\newcommand{\dint}{\displaystyle\int}

\newcommand{\one}{\mathbb{1}}
\newcommand{\zero}{\mathbb{0}}

\newcommand{\cH}{\mathcal{H}}

\newcommand{\cL}{\mathcal{L}}

\newcommand{\RR}{\mathbb{R}}
\newcommand{\CC}{\mathbb{C}}
\newcommand{\DD}{\mathbb{D}}
\newcommand{\NN}{\mathbb{N}}
\newcommand{\cB}{\mathcal{B}}
\newcommand{\cD}{\mathcal{D}}

\newcommand{\cP}{\mathcal{P}}

\newcommand{\eps}{\varepsilon}
\newcommand{\rmi}{\mathrm{i}}
\newcommand{\cc}{\mathsf{c}}

\newcommand{\conv}{\star}

\newcommand{\half}{\frac{1}{2}}
\newcommand{\nhalf}{\frac{N}{2}}

\DeclareMathOperator{\supp}{\mathrm{supp}}
\DeclareMathOperator{\dist}{\mathrm{dist}}
\DeclareMathOperator{\ran}{\mathrm{ran}}
\DeclareMathOperator{\essran}{\mathrm{ess\,ran}}
\DeclareMathOperator{\dom}{\mathrm{dom}}
\DeclareMathOperator{\spec}{\mathrm{spec}}
\DeclareMathOperator{\specp}{\mathrm{spec}_\mathrm{p}}
\DeclareMathOperator{\tr}{\mathrm{tr}}
\DeclareMathOperator{\pv}{\mathrm{p.v.}}

\newtheorem{theorem}{Theorem}[section]
\newtheorem{corollary}[theorem]{Corollary}
\newtheorem{lemma}[theorem]{Lemma}
\theoremstyle{definition}
\newtheorem{definition}[theorem]{Definition}

\newtheorem{remark}[theorem]{Remark}
\newtheorem{proposition}[theorem]{Proposition}

\renewcommand{\Tilde}{\widetilde}
\renewcommand{\Hat}{\widehat}

\newcommand{\dd}{\mathrm{d}}

\begin{document}
	\title[]{\large MIT bag model and infinite mass limit\\[\smallskipamount] in~non-smooth domains}
	
	\author[B. Benhellal]{Badreddine Benhellal}
	\address{\footnotesize  National Higher School of Mathematics, Scientific and Technology Hub of Sidi Abdellah, P.O.~Box~75,  16093~Algiers, Algeria}
	\email{badreddine.benhellal@nhsm.edu.dz}

	\author[N. K\"orner]{Noah K\"orner}
	\address{\footnotesize  Carl von Ossietzky Universit\"at Oldenburg,
		Fakult\"at V -- Mathematik und Naturwissenschaften,  Institut f\"ur Mathematik, Ammerl\"ander Heerstr. 114--118, 26129~Oldenburg (Oldb), Germany}
	\email{noah.koerner@uol.de}

	\author[D. Machado]{Dylan Machado}
	\address{\footnotesize  POEMS, CNRS, Inria, ENSTA, Institut Polytechnique de Paris,  828~Bd.~des~Mar\'echaux, 91120~Palaiseau, France} \email{dylan.machado@ensta.fr}
	
	\author[K. Pankrashkin]{Konstantin Pankrashkin}
	\address{\footnotesize Carl von Ossietzky Universit\"at Oldenburg,
		Fakult\"at V -- Mathematik und Naturwissenschaften,  Institut f\"ur Mathematik, Ammerl\"ander Heerstr. 114--118, 26129~Oldenburg (Oldb), Germany}
	\email{konstantin.pankrashkin@uol.de}

	\date{}

	\begin{abstract}
		The work is devoted to the study of Dirac operators with MIT bag boundary conditions in Euclidean domains with compact Lipschitz boundaries in arbitrary dimensions. It is shown that such operators are self-adjoint on suitable definition domains and can be recovered as the norm-resolvent limits of Dirac operators
		in the whole space with a large mass term outside the domain, under the assumption that an associated Robin-Laplacian eigenvalue has a prescribed asymptotic behavior with respect to a parameter in the boundary condition. This assumption is shown to hold for a class of non-smooth domains, which includes convex domains and, more generally, domains that can be ``locally convexified'' by suitable diffeomorphisms. To the best of our knowledge, this represents the first infinite mass interpretation for the MIT bag model in the sense of resolvent convergence
		for non-smooth domains. Most results are extended  to the generalized MIT bag boundary conditions
		with the help of the recently established congruence transform.
	\end{abstract}
	
	\maketitle

	\tableofcontents

	\section{Introduction}

	Let $n\geq 2$. For a choice of $N\in\NN$, let $\alpha_1,\dots,\alpha_n,\beta$ be pairwise anti-commuting Hermitian matrices
	with $\alpha_j^2=\beta^2=I_N$, where $I_N$ is the $N\times N$ identity matrix. We recall that due to the general theory of Clifford algebras the above choice is possible if and only if $N$ is a multiple of $2^{\lfloor\frac{n+1}{2}\rfloor}$, see e.g. \cite[Ch.~16]{lounesto},  and we refer to \cite[Ch.~15]{dg} or \cite[App.~E]{wit} for possible iterative constructions of $\alpha_j$ and $\beta$.
	
	For a parameter $m\in\RR$, the associated $n$-dimensional Dirac operator with mass $m$ is the first-order matrix differential operator acting on vector functions $f:\RR^n\to\CC^N$ by
	\[
	D_m:\  f\mapsto -\rmi \sum_{j=1}^n \alpha_j \partial_j f + m\beta f.
	\]
	In the present work, we are interested in operators given by the differential expression $D_m$ on Euclidean domains with special boundary conditions.
	
	Namely, for a Lipschitz domain $\Omega\subset\RR^n$ with compact boundary and outer unit normal $\nu$ denote by $A^\Omega_m$ the operator acting as $A^\Omega_m:\ f\mapsto D_m f$
	on the vector functions $f$ defined on $\Omega$ and satisfying the so-called \emph{MIT bag boundary condition}
	\begin{equation}
		\label{fbf1}
		f= \cB f \text{ on }\partial\Omega,\qquad  \cB:=-\rmi \beta (\alpha\cdot\nu),
	\end{equation}
	where we employ the usual notation (Clifford multiplication)
	\[
	\alpha\cdot x:=\sum_{j=1}^n x_j \alpha_j \text{ for }x=(x_1,\dots,x_n)\in\RR^n.
	\]
	The values of $f$ at the boundary in \eqref{fbf1} are understood in the sense of Sobolev traces, and the precise regularity of the functions $f$ in the operator domain is one of the main points discussed below. The operator $A^\Omega_m$ is often referred to as the \emph{MIT bag model} in $\Omega$, which was introduced in the physics literature as a model of quark confinement in hadrons~\cite{bogol,chod,degrand} and later found numerous other applications~\cite{baym,john}. It was observed in~\cite{BM} that the MIT bag model appears as an effective operator when considering the whole space with a large mass term outside the domain, which justifies the fact that the MIT bag boundary condition is often called \emph{infinite mass boundary condition}.
	
	The mathematically rigorous study of the MIT bag model seems more recent and has become particularly active during the last decade. One of the first questions to address is the rigorous definition of the above operator leading to a self-adjoint operator in the Hilbert space $L^2(\Omega,\CC^N)$. If $\Omega$ is a bounded $C^\infty$-smooth domain, the situation is covered by the general theory of elliptic differential operators, in which the MIT bag condition also appears under the name \emph{CHI boundary condition}, see \cite{baer,baer2,grosse}. It follows that the operator $A^\Omega_m$ becomes self-adjoint if considered on the ``natural'' domain
	\[
	\big\{ f\in H^1(\Omega,\CC^N):\ f=\cB f\text{ on } \partial\Omega\big\},
	\]
	which will be referred to as \emph{$H^1$-domain} later on. The result extends to bounded $C^2$ domains \cite{ALTR,BHM,SADirac2D,OBV,rab} and to a class of smooth domains with non-compact boundaries~\cite{rab2}.
	However, the situation changes if $\Omega$
	is non-smooth. The papers \cite{LTOB,PV} demonstrated that in the case of a planar piecewise smooth $\Omega$
	with concave corners the operator $A^\Omega_m$ is \emph{not} self-adjoint but just symmetric. However, it becomes self-adjoint if considered on the larger domain
	\begin{equation}
		\label{hhalf}
		\big\{ f\in H^\half(\Omega,\CC^N):\ D_m f\in L^2(\Omega,\CC^N),\  f=\cB f\text{ on } \partial\Omega\big\},
	\end{equation}
	which we call \emph{$H^\half$-domain} in what follows. The subsequent  works \cite{bb-jmp,BenhellalPankrashkin,BoundaryTripleWeylfunctions}
	showed, for $n\in\{2,3\}$, that the $H^\half$-domain is a self-adjointness domain of $A^\Omega_m$ for general Lipschitz $\Omega$ with compact boundaries, and some planar domains with non-compact boundaries and corners were covered in~\cite{CP25}. The recent paper \cite{pankrashkin2025mitbagnonsmoothconvex} demonstrated that, however, the $H^1$-domain is sufficient to guarantee the self-adjointness for bounded \emph{convex} $\Omega$ in all dimensions.
	
	In order to discuss the infinite mass interpretation, for $m,M\in\RR$ let us consider the self-adjoint operator $A^\Omega_{m,M}$ in $L^2(\RR^n,\CC^N)$ given by
	\[
	A^\Omega_{m,M}:\ f\mapsto D_0 f +(m\one_{\Omega}+M\one_{\RR^n\setminus\Omega})\beta f, \quad \dom A^\Omega_{m,M}=H^1(\RR^n,\CC^N),
	\]
	which is the free Dirac operator in the whole space with mass $m$ inside $\Omega$ and mass $M$ outside $\Omega$.
	As mentioned above, the paper \cite{BM} gave physical arguments supporting the idea that the confined operator $A^\Omega_m$ arises as a suitably defined limit of $A^\Omega_{m,M}$ as $M\to +\infty$. The first mathematically rigorous interpretation was given in \cite{SV} for $n=2$ in the following form: if $\Omega\subset \RR^2$ is a $C^2$-smooth bounded domain, then for each fixed $k\in\NN$ and $m\in\RR$ one has
	\begin{equation}
		\label{eigconv}
		E_k\big((A^\Omega_m)^2\big)=\lim_{M\to+\infty}E_k\big((A^\Omega_{m,M})^2\big),
	\end{equation}
	where $E_k(T)$ denotes the $k$-th eigenvalue of a lower semibounded operator $T$. The result was later extended to $n=3$ and bounded $C^2$-smooth domains in \cite{ALTMR} and to arbitrary $n$ and bounded $C^\infty$-smooth domains in \cite{mobp}.
	We also mention the work \cite{flam} proving similar results for Dirac operators on spin manifolds, as well as the recent extension of the convergence~\eqref{eigconv} for bounded piecewise smooth convex domains with bounded mean curvature in~\cite{pankrashkin2025mitbagnonsmoothconvex}. The paper \cite{barb} upgraded the eigenvalue convergence to the norm-resolvent convergence by showing that for any fixed $m\in\RR$ and $z\in\CC\setminus\RR$ one has
	\begin{equation}
		\label{resconv}
		r(M):=\big\|(A^\Omega_{m,M}-z)^{-1}-(A^\Omega_m-z)^{-1}\oplus\zero\big\|\xrightarrow{M\to+\infty}0,
	\end{equation}
	where the symbol $\oplus$ means the direct sum decomposition with respect to the identification $L^2(\RR^n,\CC^N)\simeq L^2(\Omega,\CC^N)\oplus L^2(\RR^n\setminus\Omega,\CC^N)$. In fact, the paper \cite{barb} considered $C^2$-smooth domains, with either compact
	or suitable controlled non-compact boundaries, and established the rate of convergence $r(M)=O(M^{-\half})$. The paper \cite{bbz} extended the result to the case of $C^\infty$-smooth bounded domains for $n=3$ with more detailed estimates
	of various terms with the help of microlocal analysis. The convergence \eqref{resconv} is often referred to as generalized norm-resolvent convergence, and it implies the convergence of eigenvalues, spectral projectors and other related quantities,
	see \cite[Satz~9.28]{weid} as well as the paper~\cite{post} for a more recent discussion. As noted in \cite{pankrashkin2025mitbagnonsmoothconvex}, the existing proofs of the norm-resolvent convergence explicitly use the boundary smoothness because they involve identities and inequalities containing the mean curvature, and we are not aware of any earlier work dealing with the infinite mass interpretation of the MIT bag model in the context of resolvent convergence 	for non-smooth domains.

	We further mention numerous recent works devoted to estimating the spectrum of $A^\Omega_m$ in terms of the geometric properties of $\Omega$, see e.g. \cite{beng-dots, bor, hlv, hobp, LOB, raulot, ALTR1}, and to spectral isoperimetric problems~\cite{ABLOB,ANT,bfhsl,bk,duran1,duranmas}.

	In view of the above overview, let us now describe the contributions of the present work and the structure of the paper. Let $\Omega\subset\RR^n$ be an open set with compact Lipschitz boundary and denote
	\[
	\nu:=\text{ the outer unit normal},\quad
	\Omega^\cc:=\RR^n\setminus\overline{\Omega}.
	\]
	In Section~\ref{sec3}, we
	show that the MIT bag operator $A^\Omega_m$ is self-adjoint on the $H^\half$-domain \eqref{hhalf}. The overall approach
	follows the same lines as for $n\in\{2,3\}$ in \cite{bb-jmp,BenhellalPankrashkin,BoundaryTripleWeylfunctions}, i.e.
	one shows first the self-adjointness of a special singular perturbation (Lorentz $\delta$-type) of the free Dirac operator supported by $\partial\Omega$ (Theorem~\ref{thm-sa}), and by looking at specific coupling constants
	one recovers the self-adjointness of the direct sum of the MIT bag operators in $\Omega$ and its complement (Corollary~\ref{lem10}). At the same time, our proof presents some new technical ingredients: while the preceding works for $n\in\{2,3\}$ used an explicit form of the resolvent integral kernel, we mostly proceed with arguments based on the theory of distributions, which significantly simplifies some steps and avoids the use of special functions and asymptotic expansions (see Section~\ref{sec2}).
	
	Section~\ref{sec4} is dedicated to establishing the resolvent convergence, for which an additional
	assumption is imposed: we require the existence of a function $\rho:[0,+\infty)\to[0,+\infty)$ with $\rho(M)=o(M^2)$
	for $M\to+\infty$ such that for all $M>0$ one has
	\begin{equation}
		\Lambda(\Omega^\cc,M):=\inf_{f\in H^1(\Omega^\cc),\,f\ne 0}\dfrac{\dint_{\Omega^\cc} |\nabla f|^2\dd x-M\dint_{\partial \Omega} |f|^2\dd S}{\dint_{\Omega^\cc} |f|^2\dd x}\ge -M^2-\rho(M).
		\label{robinineq}
	\end{equation}
	This assumption allows us to obtain the resolvent convergence \eqref{resconv} with the remainder
	\[
	r(M):=O\Big(\dfrac{\sqrt{\rho(M)}+\sqrt{M}}{M}\Big),
	\]
	see Theorem~\ref{thm19}. The proof strategy is mainly based on a representation of the resolvent using Poincar\'e-Steklov operators on $\partial\Omega$ similar to the observations made in Section~4.6 of the thesis \cite{bbthesis} for bounded $C^2$-domains, which are in turn a restructured version of~the constructions in~\cite{barb}. Our main contribution consists in showing that all objects can be properly defined and suitably estimated for a class of non-smooth domains as well.

	It should be noted that the left-hand side of \eqref{robinineq} is the bottom of the spectrum of the Robin Laplacian $f\mapsto-\Delta f$ in $\Omega^\cc$ with the Robin boundary condition $\partial_\nu f+Mf=0$, see \cite{bfk,lp}, and it is known that, in general, the inequality \eqref{robinineq} \emph{fails}: for any $c>0$ one can find $\Omega$ with $\Lambda(\Omega^\cc,M)\le -cM^2$ for large $M$. We refer to~\cite{bfk,kobp} for a detailed review on parameter-dependent Robin Laplacians. As a result,
	our arguments \emph{do not allow} to cover all Lipschitz $\Omega$. However, in Section~\ref{sec5} we find several classes of $\Omega$ for which \eqref{robinineq} holds.  If $\Omega$ is \emph{locally $C^1$-convexifiable} (i.e. coincides, up to a $C^1$-diffeomorphism, with a convex domain near each boundary point, see Definition~\ref{deff}), which includes $C^1$-domains and images of convex domains under diffeomorphisms of the whole space, then the resolvent convergence~\eqref{resconv} holds with the weak estimate $r(M)=o(1)$, see Theorem~\ref{thmgen}. In the particular case of \emph{locally convex-congurent} $\Omega$ (which includes e.g. finite unions of disjoint convex  and $C^{1,1}$-domains: see Definition~\ref{deflocconv})
	 we recover \eqref{resconv} with the error estimate $r(M)=O(M^{-\half})$, see Theorem~\ref{thmlocconvex}. To our best knowledge, these are the first results establishing the infinite mass limit in the norm-resolvent sense for non-smooth $\Omega$, and the associated estimates of $\Lambda(\Omega^\cc,M)$ in the convex and $C^1$-convexifiable cases (stated in Lemma~\ref{lem22convex} and~\ref{lemmpseudo} respectively) are new in the context of Robin eigenvalue problems as well. 
	
	In the last section~\ref{sec6} we use the results of the recent work~\cite{duran} to extend the
	self-adjointness and infinite mass limit results to the
	case of generalized MIT bag boundary conditions, see Theorem~\ref{thm62} and \ref{thm63}.
	Note that this represents the first infinite mass interpretation of generalized MIT bag models
	in non-smooth domains. The paper is concluded by a discussion of further possible extensions of the results and the techniques employed.

	Let us now briefly explain the notation. As already mentioned, everywhere below $\Omega$ stands for an open subset of $\RR^n$
	with compact Lipschitz boundary (while $\Omega$ itself can be unbounded) and $\nu$ is the outer unit normal on the boundary. As many constructions deal
	simultaneously with $\Omega$ and its complement, we denote
	\[
	\Sigma:=\partial\Omega,\quad \Omega_+:=\Omega,\quad \Omega_-:=\RR^n\setminus\overline{\Omega_+},
	\]
	and we frequently use the identification $L^2(\RR^n)\simeq L^2(\Omega_+)\oplus L^2(\Omega_-)$
	by $f\simeq(f_+,f_-)$ with $f_\pm$ being the restriction of $f$ to $\Omega_\pm$.
	We will consider the MIT bag operators associated with both $\Omega_+$ and $\Omega_-$,
	\[
	A^\pm_m:=A^{\Omega_\pm}_m \text{ in }L^2(\Omega_\pm,\CC^N),
	\]
	which are rigorously defined in \eqref{eqmit00} below. 	The matrix function $\cB$ on $\Sigma$ will always be defined by \eqref{fbf1}.

	\section{Free Dirac operator}\label{sec2}

	Let us start with some basic properties of the free Dirac operator $\DD_m$ with mass $m$ in $\RR^n$, which is a linear operator in $L^2(\RR^n,\CC^N)$ acting as
	\[
	\dom \DD_m :=H^1(\RR^n,\CC^N),
	\quad \DD_m f:=D_m f.
	\]
	We prefer to give a short self-contained proof of Theorem~\ref{thm1} below: while earlier proofs are available, see e.g. \cite[Prop.~A.1]{mobp}, we would like to simplify the available technically demanding arguments by using some algebraic constructions that are independent of the explicit choice of the matrices $\alpha_j$ and $\beta$.
	
	\begin{theorem}
		\label{thm1}
		The operator $\DD_m$ is self-adjoint and has purely essential spectrum,
		\[
		\spec \DD_m=\big(-\infty,-|m|\big]\,\cup\,\big[|m|,+\infty\big).
		\]
	\end{theorem}	
	
	\begin{proof}
		(a) Self-adjointness. As bounded symmetric perturbations do not influence the self-adjointness, assume $m=0$.
		Denote by $S$ the restriction of $\DD_0$ to $\dom S:=C^\infty_c(\RR^n,\CC^N)$, then using the usual truncations and mollifications one shows that $\DD_0\subset \overline{S}$. The theory of distributions implies that the adjoint $S^*$ of $S$ is given by 
		\[
		\dom S^*:=\big\{f\in L^2(\RR^n,\CC^N):\ D_0 f\in L^2(\RR^n,\CC^N)\big\},\quad S^*f=D_0 f.
		\]
		One has the obvious inclusion $\DD_0\subset S^*$. Let us show the reverse inclusion.
		
		Let $f\in L^2(\RR^n,\CC^N)$, then $f$ is a tempered distribution on $\RR^n$. If $f\in \dom S^*$, then the condition $D_0 f\in L^2(\RR^n,\CC^N)$ shows that the function
		$\varphi: \ \xi\mapsto (\alpha\cdot\xi)\Hat f(\xi)$
		belongs to $L^2(\RR^n,\CC^N)$. We have
		\begin{align*}
			\|\varphi\|^2_{L^2(\RR^n,\CC^N)}&=\int_{\RR^n}\big\langle (\alpha\cdot\xi)\Hat f(\xi), (\alpha\cdot\xi)\Hat
			f(\xi)\big\rangle_{\CC^N}\dd\xi\\
			&= \int_{\RR^n}\big\langle \Hat f(\xi), \underbrace{(\alpha\cdot\xi)(\alpha\cdot\xi)}_{=|\xi|^2 I_N}\Hat f(\xi)\big\rangle_{\CC^N}\dd\xi=\int_{\RR^n} |\xi|^2 \big|\Hat f(\xi)\big|^2_{\CC^N}\dd\xi.
		\end{align*}
		Therefore, the function $\xi\mapsto |\xi| \Hat f(\xi)$ belongs to $L^2(\RR^n,\CC^N)$, which in turn implies
		$f\in H^1(\RR^n,\CC^N)$, i.e. $f\in \dom\DD_0$. Hence, $S^*\subset \DD_0$, and then $\DD_0=S^*$.
		
		As the adjoint operator is always closed, from $S\subset \DD_0$ it follows that $\overline S\subset \DD_0$, and then $\overline{S}=\DD_0$. Finally, we obtain $\DD_0^*=(S^*)^*=\overline{S}=\DD_0$, i.e. $\DD_0$
		is self-adjoint.
		
		(b) Spectrum. Let $m\in\RR$ be arbitrary. Consider the matrices
		\[
		h_m(\xi):=\alpha\cdot \xi+m\beta,\quad \xi\in\RR^n.
		\]
		Let $F:L^2(\RR^n,\CC^N)\to L^2(\RR^n,\CC^N)$ be the Fourier transform (normalized to a unitary operator), then
		$F\DD_mF^{-1}=F(S^*+m\beta)F^{-1}=:H_m$,
		where $H_m$ is the self-adjoint operator in $L^2(\RR^n,\CC^N)$ acting as $H_m g(\xi):=h_m(\xi)g(\xi)$ on the domain
		\[
		\dom H_m:=\big\{ g\in L^2(\RR^n,\CC^N):\ h_m g\in L^2(\RR^n,\CC^N)\big\}.
		\]
		
		For any $\xi$ we have $h_m(\xi)^2=|\xi|^2+m^2$, hence,
		\[
		\spec h_m(\xi)\subset\big\{\,\lambda_m^-(\xi),\,\lambda_m^+(\xi)\,\big\},\qquad
		\lambda^\pm_m(\xi):=\pm \sqrt{|\xi|^2+m^2}.
		\]
		
		Using the anticommutation of $\alpha_j$ and $\beta$ and the relation $\tr (BA)=\tr(AB)$ we obtain $\alpha_j=-\beta\alpha_j\beta$, hence,
		$\tr \alpha_j=-\tr(\beta\alpha_j\beta)=-\tr(\alpha_j \beta\beta)=-\tr\alpha_j$,
		which implies $\tr \alpha_j=0$. Similarly, $\beta=-\alpha_j\beta\alpha_j$ and
		$\tr \beta=-\tr(\alpha_j \beta\alpha_j)=-\tr(\beta\alpha_j\alpha_j)=-\tr\beta$,
		therefore, $\tr \beta=0$. This yields $\tr h_m(\xi)=0$ for all $\xi$. As the trace of a Hermitian matrix is exactly
		the sum of its eigenvalues (counting the multiplicities), it follows that
		for $(\xi,m)\ne(0,0)$ each of the above numbers $\lambda_m^\pm(\xi)$ is an eigenvalue of the matrix $h_m(\xi)$ of multiplicity $\nhalf$. By standard results on parameter-dependent Hermitian matrices, see e.g. \cite[Thm.~2]{wilcox}, 
		there exist unitary $N\times N$ matrices $U_m(\xi)$ depending measurably on $\xi$ such that
		for a.e. $\xi$ one has
		\[
		U_m(\xi)^{-1}h_m(\xi)U_m(\xi)=\begin{pmatrix}
			\lambda_m^+(\xi)I_\nhalf & \zero_\nhalf\\
			\zero_\nhalf & \lambda_m^-(\xi)I_{\nhalf}
		\end{pmatrix}.
		\]
		Consider the unitary operator
		\[
		U_m: L^2(\RR^n,\CC^N)\to L^2(\RR^n,\CC^N),\quad U_mg(\xi):=U_m(\xi)g(\xi),
		\]
		then
		the operator $\Lambda_m:=U^{-1}_mH_m U_m$ decomposes as $\Lambda_m=\Lambda^+_m\oplus \Lambda^-_m$, where
		$\Lambda_m^\pm$ is the operator of pointwise multiplication by the scalar functions  $\lambda^\pm_m$ in $L^2(\RR^n,\CC^\nhalf)$, hence,
		\[
		\spec \Lambda_m^\pm=\essran \lambda_m^\pm=\pm\big[|m|,+\infty\big).
		\]
		Note that by construction the operator $\Lambda_m$ is unitarily equivalent to $\DD_m$, and it follows that
		$\spec\DD_m=\spec \Lambda_m=\spec \Lambda^+_m\cup\spec \Lambda^-_m$.
	\end{proof}

	Our next objective is to describe some properties of the integral kernel of the resolvent of $\DD_m$.

	Let $E$ denote the standard fundamental solution of $-\Delta$ in $\RR^n$,
	\[
	E(x)=\begin{cases}
		\dfrac{1}{2\pi}\log\dfrac{1}{|x|}, & n=2,\\[\bigskipamount]
		\dfrac{1}{(n-2)\omega_n |x|^{n-2}}, & n\ge 3,
	\end{cases}
	\]
	where $\omega_n$ is the hypersurface volume of the unit sphere in $\RR^n$. We have $|\xi|^2\Hat E(\xi)=\widehat{-\Delta E}=\Hat{\delta_0}=1$, i.e. $\Hat E$ coincides with the function $\xi\mapsto |\xi|^{-2}$ in $\RR^n\setminus\{0\}$. Due to the identity $D_0^2=-\Delta\otimes I_N$ the function
	\[
	G:= D_0 (E\otimes I_N),\qquad \text{i.e.}\quad G(x)=\dfrac{\rmi \alpha\cdot x}{\omega_n |x|^n},\quad x\in\RR^n\setminus\{0\},
	\]
	is a fundamental solution of the Dirac operator $D_0$ on $\RR^n$, and by construction we have
	$\Hat G(\xi)=(\alpha\cdot \xi)\Hat E(\xi)$, i.e. $\Hat G$ coincides with the function $\xi\mapsto (\alpha\cdot\xi)|\xi|^{-2}$ in $\RR^n\setminus\{0\}$. As this function is locally integrable on $\RR^n$ for any $n\ge 2$, one has
	\[
	\Hat G(\xi)=\dfrac{\alpha\cdot\xi}{|\xi|^2} \text{ for a.e. }\xi\in\RR^n.
	\]
	
	For any $m\in\RR$ and $z\in\CC\setminus\spec \DD_m$ the function
	\[
	\xi\mapsto \dfrac{1}{|\xi|^2+m^2-z^2}
	\]
	is a tempered distribution on $\RR^n$. Hence, it is the Fourier transform of some tempered distribution $E_{m,z}$ on $\RR^n$,
	and due to 
	\[
	\widehat{\delta_0}=1=(|\xi|^2+m^2-z^2)\Hat E_{m,z}=\big((-\Delta+m^2-z^2) E_{m,z}\big){}^{\Hat{}}
	\]
	the distribution $E_{m,z}$ is a fundamental solution of $-\Delta+m^2-z^2$.
	Further note that due to $(D_m-z)(D_m+z)=(-\Delta+m^2-z^2)\otimes I_N$
	the tempered distribution
	\[
	G_{m,z}:=(D_m+z)\big(E_{m,z}\otimes I_N\big)
	\]
	is a fundamental solution of $D_{m}-z$ in $\RR^n$. By construction we have
	\[
	\Hat G_{m,z}(\xi)=\dfrac{\alpha\cdot\xi+m\beta+z I_N}{|\xi|^2+m^2-z^2} \quad\text{ for a.e. }\xi\in\RR^n.
	\]
	Note that $G_{m,z}$ can be expressed in terms of special functions, but for later applications
	it is more convenient to use the direct analysis based on the following assertion.
	
	\begin{lemma}\label{lem2}
		For any $m\in\RR$ and $z\in\CC\setminus\spec\DD_m$ one has 
		$G_{m,z}=G+ F_1+F_2$, where $F_1$ and $F_2$ are tempered distributions such that $F_1$ is $C^\infty$-smooth on $\RR^n$ and $\Hat F_2$ is bounded on $\RR^n$
		with $\Hat F_2(\xi)=O(|\xi|^{-2})$ for $\xi\to\infty$.	
	\end{lemma}
	
	\begin{proof}
		Denote $F:=G_{m,z}-G$. Due to the preceding discussion, its Fourier transform $\Hat F$ coincides with the smooth function
		\[
		\xi\mapsto \dfrac{\alpha\cdot\xi+m\beta+zI_N}{|\xi|^2+m^2-z^2}-\dfrac{\alpha\cdot\xi}{|\xi|^2}
		\]
		in $\RR^n\setminus\{0\}$. Let $\chi\in C^\infty_c(\RR^n)$ such that $\chi=1$ in a neighborhood of the origin,
		then $\Hat F = \chi \Hat F + (1-\chi)\Hat F$, and it follows
		that $F=F_1 +F_2$, where $F_1$ and $F_2$ are the tempered distributions  with
		$\widehat F_1=\chi\Hat F$ and $\widehat F_2=(1-\chi)\Hat F$.
		
		Note that $\widehat F_1$ is compactly supported, hence, $F_1$ is $C^\infty$-smooth. Further, 
		\begin{align*}
			\widehat F_2(\xi)&=\big(1-\chi(\xi)\big)\Big( \dfrac{\alpha\cdot\xi+m\beta+zI_N}{|\xi|^2+m^2-z^2}-\dfrac{\alpha\cdot\xi}{|\xi|^2}\Big)\\
			&=\big(1-\chi(\xi)\big)\bigg(\dfrac{(z^2-m^2) (\alpha\cdot\xi)}{|\xi|^2\big( |\xi|^2+m^2-z^2\big)}
			+\dfrac{m\beta+zI_N}{|\xi|^2+m^2-z^2}\bigg),
		\end{align*}
		which completes the proof.
	\end{proof}
	As a direct consequence of ellipticity,  for any $m\in\RR$ and $z\in\CC\setminus\spec \DD_m$
	the resolvent of $\DD_m$,
	\[
	(\DD_m-z)^{-1}: \  H^\sigma(\RR^n,\CC^N)\ni f\mapsto G_{m,z}\conv f \in H^{\sigma+1}(\RR^n,\CC^N),
	\]
	is a bounded linear operator for any $\sigma\in\RR$ (here and below the symbol $\conv$ stands for the convolution of distributions).

	\section{Self-adjointness of the MIT bag operator}\label{sec3}

	Together with the usual Sobolev spaces $H^s$, for an open set $U\subset\RR^n$ we will consider the Dirac-Sobolev spaces
	\[
	H^s_\alpha(U,\CC^N):=\big\{f\in H^s(U,\CC^N):\ D_0 f\in L^2(U,\CC^N)\big\},\quad s\in[0,1],
	\]
	which become Hilbert spaces if equipped with the scalar products
	\[
	\langle f,g \rangle_{H^s_\alpha(U,\CC^N)}:=\langle f,g \rangle_{H^s(U,\CC^N)}+\langle D_0f,D_0g \rangle_{L^2(U,\CC^N)}.
	\]

	Recall that the Dirichlet trace operator
	\[
	C^\infty_c(\RR^n,\CC^N)\ni f\mapsto f|_\Sigma\in L^2(\Sigma,\CC^N)
	\]
	uniquely extends to a bounded linear map
	\[
	\gamma:H^1(\RR^n,\CC^N)\to H^{\half}(\Sigma,\CC^N).
	\]
	In particular,
	\begin{equation}
		\label{gamkompakt}
		\gamma:H^1(\RR^n,\CC^N)\to L^2(\Sigma,\CC^N) \text{ is compact},
	\end{equation}
	and Fubini's theorem shows that its dual with respect to the $L^2(\Sigma,\CC^N)$-scalar product is
	the bounded map
	\[
	\gamma^*: L^2(\Sigma,\CC^N)\ni h\mapsto h \delta_\Sigma \in H^{-1}(\RR^n,\CC^N),
	\]
	where $\delta_\Sigma$ is the (matrix) Dirac $\delta$-distribution supported by $\Sigma$, which is defined by
	\[
	\delta_\Sigma: \quad C^\infty_c(\RR^n,\CC^N)\ni\varphi\mapsto \int_\Sigma \varphi \,\dd S\in \CC^N,
	\]
	with $\dd S$ being the hypersurface measure. We conclude that the operator
	\begin{align}\label{Def Phi}
		\Phi_{m,z}:=\big(\gamma(\DD_m-\Bar z)^{-1}\big)^*=(\DD_m-z)^{-1}\gamma^*:L^2(\Sigma,\CC^N)\to L^2(\RR^n,\CC^N)
	\end{align}
	is bounded as well. More explicitly, one has the convolution representation
	\[
	\Phi_{m,z} h=G_{m,z}\conv (h\delta_\Sigma),\quad h\in L^2(\Sigma,\CC^N).
	\]
	
	Note that by construction one has (in the sense of distributions of $\RR^n$)
	\begin{align*}
		(D_m-z)\Phi_{m,z}h&=(D_m-z)\big(G_{m,z}\conv (h\delta_\Sigma)\big)\\
		&=\big((D_m-z)G_{m,z}\big)\conv (h\delta_\Sigma)=\delta_0\conv(h\delta_\Sigma)=h\delta_\Sigma,
	\end{align*}
	in particular,
	\[
	(D_m-z)\Phi_{m,z}h=0 \text{ in }\RR^n\setminus\Sigma \text{ for any }h\in L^2(\Sigma,\CC^N),
	\]
	which shows $\Phi_{m,z}h\in H^0_\alpha(\RR^n\setminus\Sigma,\CC^N)$ for any $h\in L^2(\Sigma,\CC^N)$. Actually,
	due to a more refined argument based on mapping properties of layer potentials for the Laplacian, see \cite[Lemma 4.2]{BoundaryTripleWeylfunctions}, the map
	\[
	\Phi_{m,z}: L^2(\Sigma,\CC^N)\to H^\half_\alpha(\RR^n\setminus\Sigma,\CC^N)
	\]
	is well-defined and bounded. We further note that by \cite[Lemma 4.1]{BoundaryTripleWeylfunctions} the maps
	\[
	C^\infty(\overline{\Omega_\pm},\CC^N)\ni f\mapsto f|_\Sigma \in L^2(\Sigma,\CC^N)
	\]
	are uniquely extended to bounded linear operators (one-sided trace maps)
	\[
	\gamma_\pm: \ H^\half_\alpha(\Omega_\pm,\CC^N)\to L^2(\Sigma,\CC^N),
	\]
	and the usual partial integration rule is valid for all $f,g\in H^\half_\alpha(\Omega_\pm,\CC^N)$,
	\begin{equation}
		\label{partint}
		\langle D_m f,g\rangle_{L^2(\Omega_\pm,\CC^N)}= \big\langle \mp\rmi (\alpha\cdot\nu) \gamma_\pm f,\gamma_\pm g\big\rangle_{L^2(\Sigma,\CC^N)}+\langle f,D_m g\rangle_{L^2(\Omega_\pm,\CC^N)}.
	\end{equation}
	Remark that the paper \cite{BoundaryTripleWeylfunctions} formally considers the cases $n\in\{2,3\}$ only, but the proofs of the above results are independent of the dimension. We further note that the actions of $\gamma$ and $\gamma_\pm$ are extended in an obvious way to functions that are $H^1$-regular in a neighborhood of $\Sigma$ and  $H^\half_\alpha$-regular in one-sided neighborhoods of $\Sigma$ only. Moreover, for a function $f$ defined on $\RR^n\setminus\Sigma$ it will be convenient to denote
	\[
	\gamma_\pm f:=\gamma_\pm f_\pm,\quad f_\pm:=f|_{\Omega_\pm}.
	\]
	
	The preceding discussion shows that the maps
	\[
	\gamma_\pm\Phi_{m,z}: \ L^2(\Sigma,\CC^N)\to L^2(\Sigma,\CC^N)
	\]
	are well-defined and bounded, so let us study them more attentively.

	\begin{lemma}\label{lem3}
		For any $m\in\RR$ and $z\in\CC\setminus\spec\DD_m$ one has
		\[
		\gamma_\pm \Phi_{m,z} h= \mp \dfrac{\rmi}{2}(\alpha\cdot\nu)h+ C_{m,z}h\quad	\text{ for all }h\in L^2(\Sigma,\CC^N),
		\]
		where $C_{m,z}=C+K_{m,z}$ with $C$ being the bounded singular integral operator in $L^2(\Sigma,\CC^N)$ given by
		the Cauchy principal value
		\[
		C h(x)=\pv \int_\Sigma \dfrac{\rmi \alpha\cdot(x-y)}{\omega_n|x-y|^n}h(y)\,\dd S(y),\quad x\in\Sigma,
		\]
		and $K_{m,z}$ is a compact operator in $L^2(\Sigma,\CC^N)$.	
	\end{lemma}
	
	\begin{proof}
		Let $h\in L^2(\Sigma,\CC^N)$. Using the representation of $G_{m,z}$ as in~Lemma~\ref{lem2}
		we obtain
		\[
		\Phi_{m,z}h=G\conv \gamma^*h + F_1\conv \gamma^*h+F_2\conv \gamma^*h=:\varphi_0+\varphi_1+\varphi_2.
		\]
		We compute $\gamma_\pm\varphi_j$ separately.
		
		Let us start with the pointwise representation
		\[
		\varphi_1(x)=(F_1\conv h\delta_\Sigma)(x)=\int_\Sigma F_1(x-y)h(y)\,\dd S(y),\qquad x\in \RR^n\setminus\Sigma.
		\]
		Due to the smoothness of $F_1$ we have $\varphi_1\in C^\infty(\RR^n,\CC^N)$, and it follows that $\gamma_\pm \varphi_1=\gamma\varphi_1=K_1h$,
		where $K_1$ is the integral operator given by
		\[
		(K_1 h)(x)=\int_\Sigma F_1(x-y)h(y)\,\dd S(y),\quad x\in\Sigma.
		\]
		The smoothness of $F_1$ implies that $K_1$ is a Hilbert-Schmidt operator in $L^2(\Sigma,\CC^N)$, in particular, compact.

		To examine $\gamma_\pm\varphi_2$ recall that due to the properties of $F_2$ stated in Lemma~\ref{lem2} the map
		\[
		H^s(\RR^n,\CC^N)\ni T\mapsto F_2\conv T \in H^{s+2}(\RR^n,\CC^N)
		\]
		is bounded for any $s\in\RR$. From $\gamma^*h\in H^{-1}(\RR^n,\CC^N)$ one obtains $\varphi_2\in H^1(\RR^n,\CC^N)$, and then $\gamma_\pm \varphi_2=\gamma \varphi_2=\gamma (F_2\conv \gamma^*h)$.
		By using \eqref{gamkompakt} we conclude that
		$\gamma_\pm\varphi_2=K_2 h$,
		where $K_2$ is a compact operator in $L^2(\Sigma,\CC^N)$.

		To compute $\gamma_\pm\varphi_0$ we note first the pointwise representation
		\[
		\varphi_0(x)=(G\conv h\delta_\Sigma)(x)=\int_\Sigma \dfrac{\rmi\alpha\cdot(x-y)}{\omega_n|x-y|^n} h(y)\,\dd S(y),\qquad x\in \RR^n\setminus\Sigma.
		\]
		From the above considerations it follows that $\varphi_0\equiv \Phi_{m,z}h-\varphi_1-\varphi_2$ is $H^\half_\alpha$-regular on each side of $\Sigma$.
		Furthermore, as $G$ is a fundamental solution of $D_0$, we have
		\[
		(\Delta\otimes I_N)\varphi_0=- D_0 (D_0 G\conv h\delta_\Sigma)=-D_0 (h\delta_\Sigma),
		\]
		in particular, $(\Delta\otimes I_N)\varphi_0=0$ in $\RR^n\setminus\Sigma$. It follows, see e.g. \cite[Thm.6.11(ii)]{bgm},
		that for a.e. $x^0\in\Sigma$ the non-tangential limits
		\[
		g_\pm(x^0):=\lim_{x\xrightarrow{\mathrm{nt}}x^0,\ x\in \Omega_\pm}\varphi_0(x)
		\]
		exist, and $\gamma_\pm \varphi_0=g_\pm$. In order to compute $g_\pm$ we remark that
		by \cite[Prop.~3.30]{hmt} for each $j\in\{1,\dots,n\}$ and almost all $x^0\in\Sigma$ one has
		\begin{align*}
			\lim_{x\xrightarrow{\mathrm{nt}}x^0,\ x\in \Omega_\pm}\int_\Sigma\, &\,\dfrac{x_j-y_j}{\omega_n|x-y|^n}h(y)\,\dd S(y)\\
			&=-\lim_{x\xrightarrow{\mathrm{nt}}x^0,\ x\in \Omega_\pm} \int_\Sigma (\partial_j E)(x-y) h(y)\,\dd S(y)\\
			&=\mp \half \nu_j(x^0)h(x^0) -\pv\int_\Sigma (\partial_j E)(x^0-y) h(y)\,\dd S(y)\\
			&=\mp \half \nu_j(x^0)h(x^0) +\pv\int_\Sigma\,\dfrac{x^0_j-y_j}{\omega_n|x^0-y|^n} h(y)\,\dd S(y);
		\end{align*}
		note that \cite{hmt} uses a different sign convention for $E$.
		By applying this identity to each component of $\varphi_0$ we obtain
		\[
		\gamma_\pm\varphi_0=g_\pm=\mp\dfrac{\rmi}{2}(\alpha\cdot\nu)h+ Ch.
		\]
		
		Therefore,
		\begin{align*}
			\gamma_\pm \Phi_{m,z} h&=\gamma_\pm\varphi_0+\gamma_\pm \varphi_1+\gamma_\pm \varphi_2
			=\mp\dfrac{\rmi}{2}(\alpha\cdot\nu)h+ Ch +(K_1+K_2)h,
		\end{align*}
		and we arrive at the required representation with $K_{m,z}:=K_1+K_2$. The boundedness of $C$ is a standard result in the theory of singular integral operators, see Remark~\ref{remlim} below.
	\end{proof}
	
	\begin{corollary}\label{corol4}
		For any $m\in\RR$, $z\in\CC\setminus\spec\DD_m$ and $h\in L^2(\Sigma,\CC^N)$ one has
		\begin{align*}
			\dfrac{\gamma_+ +\gamma_-}{2} \Phi_{m,z} h&= C_{m,z}h, &
			(\gamma_+ -\gamma_-) \Phi_{m,z} h&= -\rmi(\alpha\cdot \nu)h.
		\end{align*}
		In particular, the operator $\Phi_{m,z}$ is injective.
	\end{corollary}
	
	\begin{proof}
		The identities follow directly by Lemma~\ref{lem3}. If $\Phi_{m,z}h=0$ for some $h$, then the second identity implies $h=0$, which shows the injectivity.
	\end{proof}

	\begin{remark}\label{remlim}
		If one considers the functions
		\[
		\rho_j:\quad \RR^n\setminus\{0\}\ni x\mapsto \dfrac{x_j}{|x|^n}\in\RR
		\]
		and, for $\eps>0$, the integral operators with truncated integral kernels,
		\[
		T_j^\eps: \ L^2(\Sigma)\to L^2(\Sigma),\quad T^\eps_j h(x)=\int_\Sigma \one_{|x-y|>\eps}\rho_j(x-y)h(y)\,\dd S(y),
		\]
		then the seminal results~\cite[Sec.~6]{gd} on singular integral operators show that the maximal operators $\Tilde T_j$,
		\[
		\Tilde T_j :L^2(\Sigma)\to L^2(\Sigma),\quad \Tilde T_j h(x):=\sup_{\eps>0}\big|T_j^\eps h(x)\big|,
		\]
		are well-defined and bounded (note that in general these operators are not linear). The results of~\cite{hmt} cited above in the proof of Lemma~\ref{lem3} show that the Cauchy principal value integrals
		\[
		T_j h(x):=\pv\int_\Sigma \rho_j(x-y)h(y)\,\dd s(y)\equiv \lim_{\eps\to 0} T_j^\eps h(x)
		\]
		exist for a.e. $x\in\Sigma$, and in view of the pointwise inequalities
		\[
		\big|T_j h(x)\big|=\big|\lim_{\eps\to 0} T_j^\eps h(x)\big|\le \sup_{\eps>0}\big|T_j^\eps h(x)\big|=\Tilde T_j h(x) 
		\]	
		the boundedness of $\Tilde T_j$ implies the boundedness of $T_j$, which in turn shows that the matrix operator $C$ is bounded.
		
		If $h\in L^2(\Sigma)$, then in view of the relations
		\[
		\lim_{\eps\to 0} \big|T_j^\eps h(x)-T_j h(x)\big|=0,\quad \big|T_j^\eps h(x)-T_j h(x)\big|\le 2\big|\Tilde T_j h(x)\big|,
		\]
		which are valid for a.~e. $x\in\Sigma$, one obtains with the help of dominated convergence theorem
		\[
		\lim_{\eps\to 0}\int_\Sigma \big|(T_j^\eps-T_j)h(x)\big|^2\dd s(x)=0,
		\]
		i.e. $T_j^\eps\xrightarrow{\eps\to 0}T_j$ strongly. By applying these observations to each entry of $C$ we obtain the following technical assertion:
	\end{remark}
	
	\begin{lemma}\label{lem5}
		For $\eps>0$ define $C^\eps:L^2(\Sigma,\CC^N)\to L^2(\Sigma,\CC^N)$ by
		\[
		C^\eps h(x)=\int_\Sigma \one_{|x-y|>\eps}\dfrac{\rmi \alpha\cdot(x-y)}{\omega_n|x-y|^n}h(y)\,\dd S(y),\quad x\in\Sigma,
		\]	
		then $C^\eps\xrightarrow{\eps\to 0}C$ strongly.
	\end{lemma}
	
	\begin{corollary}\label{corol7}
		The operator $\rmi \beta C$ is self-adjoint in $L^2(\Sigma,\CC^N)$.
	\end{corollary}
	
	\begin{proof}
		Using Lemma~\ref{lem5} we have $(\rmi \beta C^\eps)^*\xrightarrow{\eps\to 0}(\rmi \beta C)^*$ weakly. The operator $\rmi\beta C^\eps$ is an integral operator
		with the Hilbert-Schmidt integral kernel
		\[
		L_\eps(x,y):=-\dfrac{\beta \alpha\cdot(x-y)}{\omega_n|x-y|^n}\one_{|x-y|>\eps},
		\]
		and using the anti-commutation of $\alpha_j$ and $\beta$ one obtains
		\[
		L_\eps(y,x)^*=\dfrac{-\alpha\cdot(y-x)\beta}{\omega_n|y-x|^n}\one_{|y-x|>\eps}=\dfrac{-\beta \alpha\cdot(x-y)}{\omega_n|x-y|^n}\one_{|x-y|>\eps}=L_\eps(x,y),
		\]
		which shows that each $\rmi\beta C^\eps$ is self-adjoint. Hence,
		\[
		(\rmi \beta C)^*=\text{weak\,}\lim_{\eps\to 0}(\rmi \beta C^\eps)^*\equiv \text{weak\,}\lim_{\eps\to 0}\rmi \beta C^\eps
		=\text{strong\,}\lim_{\eps\to 0}\rmi \beta C^\eps=\rmi\beta C. \qedhere
		\]
	\end{proof}

	For $m,\mu\in\RR$ consider the following operator $\DD_{m,\mu}$ in $L^2(\RR^n,\CC^N)$:
	\begin{equation}
		\label{dirac-mu}
		\begin{aligned}
			\DD_{m,\mu} f&:=D_m f \text{ in }\RR^n\setminus\Sigma,\\
			\dom \DD_{m,\mu}&:=\Big\{
			f\in H^\half_\alpha(\RR^n\setminus\Sigma,\CC^N):\\
			&\quad \qquad \rmi (\alpha\cdot\nu)(\gamma_+f-\gamma_-f)+\mu\beta\dfrac{\gamma_+f + \gamma_-f}{2} =0
			\Big\}.	
		\end{aligned}
	\end{equation}
	The operator $\DD_{m,\mu}$ is often formally written as $\DD_m+\mu\beta\delta_\Sigma$ and interpreted as a Lorentz $\delta$-perturbation of strength $\mu$ supported on $\Sigma$, see \cite{BoundaryTripleWeylfunctions}.
	In view of the partial integration rule \eqref{partint} we have for any $f,g\in\dom\DD_{m,\mu}$:
	\begin{align*}
		\langle \DD_{m,\mu}&  f,g\rangle_{L^2(\RR^n,\CC^N)}-\langle f,\DD_{m,\mu} g\rangle_{L^2(\RR^n,\CC^N)}\\
		&=\langle D_m f, g\rangle_{L^2(\Omega_+,\CC^N)}-\langle f,D_m g\rangle_{L^2(\Omega_+,\CC^N)}\\
		&\phantom{=}+\langle D_m f, g\rangle_{L^2(\Omega_-,\CC^N)}-\langle f,D_m g\rangle_{L^2(\Omega_-,\CC^N)}\\
		&=-\Big\langle \rmi (\alpha\cdot\nu) \gamma_+ f,\gamma_+ g\Big\rangle_{L^2(\Sigma,\CC^N)}
		+\Big\langle \rmi (\alpha\cdot\nu) \gamma_- f,\gamma_- g\Big\rangle_{L^2(\Sigma,\CC^N)}\\
		&=-\Big\langle \rmi (\alpha\cdot\nu) \gamma_+ f,\dfrac{\gamma_+ g+\gamma_- g}{2}\Big\rangle_{L^2(\Sigma,\CC^N)}-\Big\langle \rmi (\alpha\cdot\nu) \gamma_+ f,\dfrac{\gamma_+ g-\gamma_- g}{2}\Big\rangle_{L^2(\Sigma,\CC^N)}\\
		&\phantom{=}+\Big\langle \rmi (\alpha\cdot\nu) \gamma_- f,\dfrac{\gamma_+ g+\gamma_- g}{2}\Big\rangle_{L^2(\Sigma,\CC^N)}-\Big\langle \rmi (\alpha\cdot\nu) \gamma_- f,\dfrac{\gamma_+ g-\gamma_- g}{2}\Big\rangle_{L^2(\Sigma,\CC^N)}\\
		&=\Big\langle -\rmi (\alpha\cdot\nu)(\gamma_+f - \gamma_- f),\dfrac{\gamma_+ g+\gamma_- g}{2}\Big\rangle_{L^2(\Sigma,\CC^N)}\\
		&\phantom{=}+
		\Big\langle -\rmi (\alpha\cdot\nu)\dfrac{\gamma_+f +\gamma_- f}{2},\gamma_+ g-\gamma_- g\Big\rangle_{L^2(\Sigma,\CC^N)}\\
		&=\Big\langle \mu\beta\dfrac{\gamma_+f + \gamma_-f}{2},\dfrac{\gamma_+ g+\gamma_- g}{2}\Big\rangle_{L^2(\Sigma,\CC^N)}\\
		&\phantom{=}+\Big\langle -\rmi (\alpha\cdot\nu)\dfrac{\gamma_+f +\gamma_- f}{2},
		\rmi(\alpha\cdot\nu)\mu\beta\dfrac{\gamma_+g + \gamma_-g}{2}
		\Big\rangle_{L^2(\Sigma,\CC^N)}\\
		&=\Big\langle \mu\beta\dfrac{\gamma_+f + \gamma_-f}{2},\dfrac{\gamma_+ g+\gamma_- g}{2}\Big\rangle_{L^2(\Sigma,\CC^N)}\\
		&\phantom{=}+\Big\langle -\mu \beta \dfrac{\gamma_+f +\gamma_- f}{2},
		\dfrac{\gamma_+g + \gamma_-g}{2}
		\Big\rangle_{L^2(\Sigma,\CC^N)}=0,
	\end{align*}
	i.e. $\DD_{m,\mu}$ is symmetric. We are now going to show that $\DD_{m,\mu}$ is self-adjoint.

	\begin{lemma}\label{lem-prime}
		Let $f\in H^\half_\alpha(\RR^n\setminus\Sigma)$ and $z\in\CC$ with $(D_m-z)f=0$ in $\RR^n\setminus\Sigma$, then in $\cD'(\RR^n,\CC^N)$ one has
		\[
		(D_m-z)f=h\delta_\Sigma,\quad h:=\rmi (\alpha\cdot\nu)(\gamma_+ f-\gamma_-f)\in L^2(\Sigma,\CC^N).
		\]
	\end{lemma}	
	
	\begin{proof}
		The
		partial integration rule \eqref{partint} shows that for any
		$\varphi\in C^\infty_c(\RR^n,\CC^N)$ one has
		\begin{align*}
			\big\langle f,&(D_m-\Bar z)\varphi\big\rangle_{L^2(\RR^n,\CC^N)}
			=\big\langle f,(D_m-\Bar z)\varphi\big\rangle_{L^2(\Omega_+,\CC^N)}
			+\big\langle f,(D_m-\Bar z)\varphi\big\rangle_{L^2(\Omega_-,\CC^N)}\\
			&=\big\langle \gamma_+f,-\rmi (\alpha\cdot\nu)\gamma\varphi\big\rangle_{L^2(\Sigma,\CC^N)}+\big\langle \underbrace{(D_m-z)f}_{=0},\varphi\big\rangle_{L^2(\Omega_+,\CC^N)}\\
			&\qquad+\big\langle \gamma_-f,\rmi(\alpha\cdot\nu)\gamma\varphi\big\rangle_{L^2(\Sigma,\CC^N)}+\big\langle \underbrace{(D_m-z)f}_{=0},\varphi\big\rangle_{L^2(\Omega_-,\CC^N)}\\
			&=\int_\Sigma \big\langle \rmi (\alpha\cdot\nu)(\gamma_+ f-\gamma_- f),\gamma \varphi\big\rangle_{\CC^N}\dd S. \qedhere
		\end{align*}	
	\end{proof}

	\begin{lemma}\label{lem8} For any $m,\mu\in\RR$ and $z\in\CC\setminus\spec \DD_m$ one has
		\begin{equation}
			\label{kerker}
			\ker(\DD_{m,\mu}-z)=\Phi_{m,z}\ker \Lambda_{m,z,\mu},\qquad \Lambda_{m,z,\mu}:=I+\mu\beta C_{m,z}.
		\end{equation}
	\end{lemma}
	
	\begin{proof}
		(a) Let $f:=\Phi_{m,z}h$ with $h\in L^2(\Sigma,\CC^N)$. 
		One has $f\in H^\half_\alpha(\RR^n\setminus\Sigma,\CC^N)$ with $(D_m-z)f=0$ in $\RR^n\setminus\Sigma$.
		Hence, one has $f\in \dom \DD_{m,\mu}$ if and only if $f$ satisfies the transmission condition at $\Sigma$.
		By Corollary~\ref{corol4} we have
		\begin{multline*}
			\rmi (\alpha\cdot\nu)(\gamma_+f-\gamma_-f)+\mu\beta\dfrac{\gamma_+f + \gamma_-f}{2}\\
			=\rmi (\alpha\cdot\nu) (-\rmi) (\alpha\cdot\nu)h+\mu \beta C_{m,z}h\equiv h + \mu\beta C_{m,z}h
			\equiv \Lambda_{m,z,\mu}h.
		\end{multline*}
		
		(b) The inclusion $\supset$ in \eqref{kerker} follows directly by (a).
		
		(c) Let $f\in \ker (\DD_{m,\mu}-z)$. Then, in particular, $f\in H^\half_\alpha(\RR^n\setminus\Sigma)$
		with $(D_m-z)f=0$ in $\Omega_\pm$, and Lemma~\ref{lem-prime} shows that $(D_m-z)f=h\delta_\Sigma$ in the sense of distributions on $\RR^n$, with the function $h:=\rmi (\alpha\cdot\nu)(\gamma_+ f-\gamma_-f)\in L^2(\Sigma,\CC^N)$.
		It follows that
		\begin{align*}
			f&=\delta_0\conv f=\big((D_m-z)G_{m,z}\big)\conv f=(D_m-z)( G_{m,z}\conv f)\\
			&=G_{m,z}\conv (D_m-z)f=G_{m,z}\conv h\delta_\Sigma=\Phi_{m,z}h,
		\end{align*}
		and the computation in (a) implies $h\in \ker\Lambda_{m,z,\mu}$.
	\end{proof}
	
	\begin{lemma}\label{lem10new}
		For any $m,\mu\in\RR$ and any $z\in\CC\setminus\RR$ the operator $\Lambda_{m,z,\mu}$ defined in \eqref{kerker} is an isomorphism of $L^2(\Sigma,\CC^N)$.
	\end{lemma}
	
	\begin{proof}
		(a) For $\mu\ne 0$ we have
		\[
		\Lambda:=I+\mu \beta C= -\rmi \mu\big(\dfrac{\rmi}{\mu}+\rmi\beta C\big),
		\]
		and in view of the self-adjointness of $\rmi\beta C$ (Corollary~\ref{corol7}) the operator $\Lambda$ has a bounded inverse, in particular, it is a zero-index Fredholm operator. Recall that $\Lambda_{m,z,\mu}=\Lambda +\mu\beta K_{m,z}$, where $K_{m,z}$ is a compact operator (Lemma~\ref{lem3}), which means that
		\[
		\Lambda_{m,z,\mu} \text{ is a zero-index Fredholm operator.}
		\]
	The conclusion is immediate for $\mu=0$ too.
		
		(b) The operator $\DD_{m,\mu}$ is symmetric, hence, $\DD_{m,\mu}-z$ is injective, and Lemma~\ref{lem8}
		shows that $\Lambda_{m,z,\mu}$ is also injective. Due to the conclusion in (a),
		\[
		\Lambda_{m,z,\mu}: L^2(\Sigma,\CC^N)\to L^2(\Sigma,\CC^N) \text{ is bijective}.\qedhere
		\]
	\end{proof}
	
	\begin{theorem}\label{thm-sa}
		For any $m,\mu\in\RR$ the operator $\DD_{m,\mu}$ is self-adjoint.
	\end{theorem}
	
	\begin{proof}
		As $\DD_{m,\mu}$ is symmetric, it suffices to show $\ran(\DD_{m,\mu}-z)=L^2(\RR^n,\CC^N)$
		for every $z\in\CC\setminus\RR$.	 So let us pick any $z\in\CC\setminus\RR$ and $g\in L^2(\RR^n,\CC^N)$. In view of Lemma~\ref{lem10new} one may define
		\begin{equation}
			\label{fg}
			f:=(\DD_m-z)^{-1}g-\mu\Phi_{m,z}\Lambda_{m,z,\mu}^{-1}\beta\Phi_{m,\Bar z}^*g.
		\end{equation}
		We are going to show that $f\in \dom \DD_{m,\mu}$ with $(\DD_{m,\mu}-z)f=g$, which will conclude the proof.
		
		Due to the definition and mapping properties of the operators in \eqref{fg} the function $f$ belongs to $H^\half_\alpha(\RR^n\setminus\Sigma)$ and satisfies
		$(D_m-z)f=g$ in $\RR^n\setminus\Sigma$. It remains to show that $f$ satisfies the transmission condition on $\Sigma$.
		We have, using \eqref{Def Phi},
		\begin{align*}
			\gamma_\pm f&=\gamma(\DD_m-z)^{-1}g-\mu\gamma_\pm \Phi_{m,z}\Lambda_{m,z,\mu}^{-1}\beta\Phi_{m,\Bar z}^*g\\
			&=\Phi_{m,\Bar z}^*g -\mu\big(\mp\dfrac{\rmi}{2}(\alpha\cdot\nu) +C_{m,z}\big)\Lambda_{m,z,\mu}^{-1}\beta\Phi_{m,\Bar z}^*g\\
			&=\Big[\beta-\mu\big(\mp\dfrac{\rmi}{2}(\alpha\cdot\nu) +C_{m,z}\big)\Lambda_{m,z,\mu}^{-1}\Big] h,\qquad h:=\beta\Phi_{m,\Bar z}^*g,
		\end{align*}
		therefore,
		\begin{align*}
			\dfrac{\gamma_+ f +\gamma_- f}{2}&=(\beta-\mu C_{m,z}\Lambda_{m,z,\mu}^{-1})h, &
			\gamma_+ f -\gamma_- f&=\mu \rmi (\alpha\cdot\nu)\Lambda_{m,z,\mu}^{-1}h,
		\end{align*}
		and finally
		\begin{align*}
			\rmi (\alpha\cdot\nu)&(\gamma_+f -\gamma_-f)+\mu\beta\dfrac{\gamma_+f + \gamma_-f}{2}
			=-\mu \Lambda_{m,z,\mu}^{-1}h+\mu\beta(\beta-\mu C_{m,z}\Lambda_{m,z,\mu}^{-1})h\\
			&=-\mu(\underbrace{I+\mu\beta C_{m,z}}_{\equiv \Lambda_{m,z,\mu}})\Lambda_{m,z,\mu}^{-1}h+\mu h=0. \qedhere
		\end{align*}	
	\end{proof}

The MIT bag operators $A^\pm_m$ in $\Omega_\pm$ act by $f\mapsto D_m f$ on the domains
	\begin{equation}
		\label{eqmit00}
	\dom A^\pm_m=\big\{f\in H^\half_\alpha(\Omega_\pm,\CC^N):\ \gamma_\pm f=\pm\cB\gamma_\pm f\big\}.
	\end{equation}
	The minus sign for $\Omega_-$ is due to the fact that $\nu$ is the \emph{inner} unit normal for $\Omega_-$.

	\begin{corollary}\label{lem10}
		For any $m\in\RR$ one has $\DD_{m,2}=A^{+}_m\oplus A^{-}_m$.	In particular, both $A^{\pm}_m$ are self-adjoint.
	\end{corollary}
	
	\begin{proof}
		As the action and the Sobolev regularity of the functions in the operator domain agree, it remains to check that the transmission condition in \eqref{dirac-mu} with $\mu=2$ splits into two separate boundary conditions for $f_+$ and $f_-$.
		The transmission condition in question is $\rmi (\alpha\cdot\nu)(\gamma_+f-\gamma_-f)+\beta(\gamma_+f + \gamma_-f) =0$,
		which is equivalently rewritten as
		\begin{equation}
			\label{transm}
			(I-\cB)\gamma_+f_+ + (I+\cB)\gamma_-f_-=0.
		\end{equation}
		If the MIT boundary conditions $\gamma_\pm f_\pm=\pm\cB\gamma_\pm f_\pm$
		are satisfied, then obviously \eqref{transm} is also true. On the other hand, a simple computation shows that
		\[
		(I\pm \cB)^2=2(I\pm \cB),\quad (I-\cB)(I+\cB)=(I+\cB)(I-\cB)=0.
		\]
		Hence, if \eqref{transm} is satisfied, then by applying $I\pm\cB$ to  both sides one recovers two independent
		boundary conditions $(I\mp \cB)\gamma_\pm f_\pm=0$.	
		
		The above shows the required direct sum decomposition. The operator $\DD_{m,2}$ is self-adjoint (Theorem~\ref{thm-sa}), which shows that both components $A^{\Omega_\pm}_m$ are self-adjoint too.	
	\end{proof}

	For completeness we state and prove some simple spectral enclosures (which are currently spread over the literature).
	
	\begin{lemma}\label{lem13spec} For any $m\ge 0$ one has
		$\spec A^+_m\cap (-m,m)=\emptyset$
		and, in addition, $\pm m\notin \specp A^+_m$. The same holds for $A^-_m$.
	\end{lemma}
	
	\begin{proof}
		(a) Let $f\in \dom A^+_m$, then
		\begin{equation}
			\label{eq00}
			\begin{aligned}
				\|A^+_mf&\|^2_{L^2(\Omega_+,\CC^N)}=\int_{\Omega_+} |D_0 f +m\beta f|^2_{\CC^N}\dd x\\
				&=\int_{\Omega_+} |D_0 f|^2\dd x+m^2\int_{\Omega_+} |\beta f|^2_{\CC^N}\dd x+2m\Re \int_{\Omega_+} \langle  D_0f,\beta f\rangle_{\CC^N}\,\dd x.
			\end{aligned}
		\end{equation}
		The integration by parts \eqref{partint} yields
		\begin{align*}
			\int_{\Omega_+} \langle  D_0f,\beta f\rangle_{\CC^N}\,\dd x
			&=\int_{\Sigma} \big\langle  -\rmi (\alpha\cdot\nu) \gamma_+ f,\beta \gamma_+ f\big\rangle_{\CC^N}\,\dd S
			+\int_{\Omega_+} \langle f, D_0\beta f\rangle_{\CC^N}\,\dd x\\
			&=\int_{\Sigma} \big\langle  -\rmi (\alpha\cdot\nu) \gamma_+ f,\beta \gamma_+ f\big\rangle_{\CC^N}\,\dd S
			-\int_{\Omega_+} \langle \beta f, D_0 f\rangle_{\CC^N}\,\dd x,
		\end{align*}
		therefore,
		\begin{equation*}
			2\Re \int_{\Omega_+} \langle  D_0f,\beta f\rangle_{\CC^N}\,\dd x
			=\int_{\Sigma} \big\langle  -\rmi (\alpha\cdot\nu) \gamma_+ f,\beta \gamma_+ f\big\rangle_{\CC^N}\,\dd S.
		\end{equation*}
		The MIT boundary condition for $f$ reads as $-\rmi (\alpha\cdot\nu) \gamma_+ f=\beta\gamma_+ f$,
		hence,
		\begin{align*}
			2\Re \int_{\Omega_+} \langle  D_0f,\beta f\rangle_{\CC^N}\,\dd x
			&=\int_{\Sigma} \big| -\rmi (\alpha\cdot\nu) \gamma_+ f\big|^2_{\CC^N}\,\dd S
			=\int_{\Sigma} |\gamma_+ f|^2_{\CC^N}\,\dd S.
		\end{align*}
		The substitution of the last equality into \eqref{eq00} results in
		\begin{equation}
			\begin{aligned}
				\|A^{+}_mf\|^2_{L^2({\Omega_+},\CC^N)}&=
				\int_{\Omega_+} |D_0 f|^2_{\CC^N}\dd x+m^2\int_{\Omega_+} |f|^2_{\CC^N}\dd x+m \int_{\partial\Omega} |\gamma_+ f|^2\,\dd S\\
				&\ge m^2 \|f\|^2_{L^2({\Omega_+},\CC^N)},
			\end{aligned}
			\label{amf}
		\end{equation}
		which proves the first assertion.
		
		(b) Let $m>0$ and $f\in \dom A^+_m$ with $A^+_m f=\pm mf$, then \eqref{amf} yields
		$\gamma_+ f=0$. Let $\Tilde f$ be the extension of $f$ by zero to $\RR^n$, then for any $\varphi\in C^\infty_c(\RR^n,\CC^N)$ one has with the help of the partial integration
		\begin{align*}
			\langle \Tilde f, D_m \varphi\rangle_{L^2(\RR^n,\CC^N)}&=\int_{\Sigma} \big\langle \underbrace{\gamma_+ f}_{=0},-\rmi(\alpha\cdot\nu) \gamma_+ \varphi \big\rangle_{\CC^N}\dd S
			+\int_{\Omega_+} \big\langle D_m f,\varphi\big\rangle_{\CC^N}\dd x\\
			&=\langle \widetilde {D_m f}, \varphi\rangle_{L^2(\RR^n,\CC^N)},
		\end{align*}
		with $\widetilde {D_m f}$ being the extension of $D_m f$ by zero to the whole of $\RR^n$. This means that
		$D_m \Tilde f=\widetilde {D_m f}$ in $\cD'(\RR^n,\CC^N)$. Hence, $\Tilde f\in L^2(\RR^n,\CC^N)$ with $D_m\Tilde f\in L^2(\RR^n,\CC^N)$, and the global elliptic regularity in $\RR^n$ shows the inclusion $\Tilde f\in H^1(\RR^n,\CC^N)$, and then $f\in H^1_0(\Omega,\CC^N)$. In addition, in $\cD'(\Omega,\CC^N)$ we have
		\[
		(-\Delta +m^2)f=D_m^2 f= A^{+}_m (A^+_m f)=m^2 f,
		\]
		i.e. $\Delta f=0$. Together with $f\in H^1_0({\Omega_+},\CC^N)$ this means that each component of $f$ lies in the kernel of
		the Dirichlet Laplacian on ${\Omega_+}$. As this kernel is trivial, we obtain $f=0$. This shows $\pm m\notin \specp A^{+}_m$ for all $m>0$.
		
		(c) It remains to show that $\ker A^+_0=\{0\}$. Assume by contradiction that there is some $f\in \ker A^+_0$ with $f\ne 0$, then for any $m\ne 0$ one has
		\begin{align*}
			\|A^+_m f\|_{L^2({\Omega_+},\CC^N)}&=\|\underbrace{A^{+}_0 f}_{=0}+m\beta f\|_{L^2({\Omega_+},\CC^N)}\\
			&=\|m\beta f\|_{L^2({\Omega_+},\CC^N)}=|m|\|f\|_{L^2({\Omega_+},\CC^N)},
		\end{align*}
		and the spectral theorem implies that $\spec A^+_m\cap \big[-|m|,|m|\big]\ne\emptyset$ for any $m\ne 0$, which
		contradicts the conclusion of (a) and (b) for $m>0$.
		
		(d) As $\Omega_+$ and $\Omega_-$ play symmetric roles, the assertions are valid for $A^-_m$ as well.
	\end{proof}

	\section{Infinite mass limit and resolvent convergence}\label{sec4}
	
	Consider the following two-mass Dirac operator $A_{m,M}$ in $L^2(\RR^n,\CC^N)$:
	\[
	A_{m,M}:= \DD_0 + (m\one_{\Omega_+}+M\one_{\Omega_-})\beta, \quad \dom A_{m,M}:=H^1(\RR^n,\CC^N).
	\]

	We further consider the resolvents, for $z$ in the respective resolvent sets,
	\[
	R^\pm_m(z):=(A^\pm_m-z)^{-1},\quad R_{m,M}(z):=(A_{m,M}-z)^{-1}.
	\]
	The following maps will be useful:
	\begin{align*}
		e_\pm:& \ L^2(\Omega_\pm,\CC^N)\to L^2(\RR^n,\CC^N) \text{ is the operator of extension by zero},\\
		r_\pm:& \ L^2(\RR^n,\CC^N)\to L^2(\Omega_\pm,\CC^N) \text{ is the operator of restriction to }\Omega_\pm.
	\end{align*}
	
	For
	\[
	\cB\equiv -\rmi\beta(\alpha\cdot\nu),
	\qquad\cP_\pm =\dfrac{I\pm\cB}{2},
	\]
	we have 
	\begin{gather*}
		\cP_\pm\cP_\mp=0,\quad \cB=\cP_+ - \cP_-, \quad I=\cP_+ + \cP_ -,\\
		\quad \cP_\pm \cB=\cB\cP_\pm =\pm \cP_\pm,
		\quad \ran \cP_\pm=\ker \cP_\mp.
	\end{gather*}
	We note that
	\begin{align*}
		\dom A^\pm_m&=\big\{f_\pm\in H^\half_\alpha(\Omega_\pm,\CC^N):\ \cP_\mp\gamma_\pm f_\pm=0\big\}\\
		&\equiv \big\{f_\pm\in H^\half_\alpha(\Omega_\pm,\CC^N):\ \gamma_\pm f_\pm\in \ran \cP_\pm\big\}.
	\end{align*}

	\begin{lemma}\label{lem14}
		Let $m\in\RR$ and $z\in\CC\setminus\spec A^{\pm}_m$. There exist constants $C_\pm>0$ such that	
		for any $\psi_\pm\in \cP_\mp L^2(\Sigma,\CC^N)$ there is a unique
		$f_\pm\in H^\half_\alpha({\Omega_\pm},\CC^N)$ with
		\begin{equation}
			\label{bcprob}
			(D_m-z)f_\pm= 0 \text{ in }\Omega_\pm,\qquad \cP_\mp\gamma_\pm f_\pm=\psi_\pm \text{ on }\Sigma,
		\end{equation}
		and it holds that
		\begin{equation}
			\label{norms}
			\|f_\pm\|_{H^\half_\alpha(\Omega_\pm,\CC^N)}\le C_\pm\|\psi_\pm\|_{L^2(\Sigma,\CC^N)}.
		\end{equation}
	\end{lemma}

	\begin{proof}
		We consider $\Omega_+$ only for better readability. The case of $\Omega_-$ then follows, as $\Omega_\pm$ play symmetric roles. Denote by $f$ the extension of $f_+$ by zero to the whole of $\RR^n$.
		
		(a) The uniqueness is clear: if there are two solutions $f_+$ and $f'_+$, then $f_+-f'_+\in \dom A^+_m$
		with $(A^+_m-z)(f_+-f'_+)=0$, and the assumption $z\notin\spec A^+_m$ implies $f_+-f'_+=0$.
		
		(b) Assume for instance that $f_+$ satisfies the required properties. Then $f\in H^\half_\alpha(\RR^n\setminus\Sigma,\CC^N)$ with $(D_m-z)f=0$ in $\RR^n\setminus\Sigma$, and Lemma~\ref{lem-prime}
		yields
		\begin{align*}
			(D_m-z)f&=h\delta_\Sigma \text{ in }\cD'(\RR^n,\CC^N),\\
			h&:=\rmi (\alpha\cdot\nu)(\gamma_+ f-\gamma_-f)\equiv \rmi(\alpha\cdot \nu)\gamma_+ f_+\equiv -\beta \cB \gamma_+ f_+\\
			&=-\beta(\cP_+ \gamma_+ f_+ - \cP_- \gamma_+ f_+)\equiv \beta(\psi_+ - \cP_+ \gamma_+ f_+).
		\end{align*}
		By taking the convolution with the fundamental solution $G_{m,z}$ we obtain
		\begin{equation}
			\label{fexpr1}
			\begin{aligned}
				f&= G_{m,z}\star \big((D_m-z)f\big)\\
				&=G_{m,z}\star \big(\beta(\psi_+ - \cP_+ \gamma_+ f_+)\delta_\Sigma\big)\equiv \Phi_{m,z}\big(\beta(\psi_+ - \cP_+ \gamma_+ f_+)\big).
			\end{aligned}
		\end{equation}
		By applying $\half(\gamma_++\gamma_-)$ to both sides (note that $\gamma_- f=0$ by construction), with the help of Corollary~\ref{corol4} we arrive at
		\begin{equation}
			\label{eq11}
			\half\gamma_+ f_+=C_{m,z}\big(\beta(\psi_+ - \cP_+ \gamma_+ f_+)\big).
		\end{equation}
		We have $\gamma_+ f_+=\cP_+\gamma_+ f_+ +\cP_- \gamma_+ f_+=\cP_+\gamma_+ f_+ + \psi_+$,
		and \eqref{eq11} takes the form
		\begin{equation}
			\label{eq22}
			\big(\half -C_{m,z}\beta\big)\psi_+= -\big(\half +C_{m,z}\beta\big)\cP_+ \gamma_+ f_+.
		\end{equation}
		Note that we have
		\[
		\half +C_{m,z}\beta=\big(\half \beta +C_{m,z}\big)\beta=\half\beta(I+2\beta C_{m,z})\beta=\half\beta \Lambda_{m,z, 2}\beta,
		\]
		see \eqref{kerker}, which is a bijective operator in $L^2(\Sigma,\CC^N)$ due to Lemma~\ref{lem10new}.	The identity \eqref{eq22}
		implies then
		\[
		\cP_+ \gamma_+ f_+=-\Big(\half +C_{m,z}\beta\Big)^{-1}\Big(\half -C_{m,z}\beta\Big)\psi_+\equiv \psi_+ -\Big(\half +C_{m,z}\beta\Big)^{-1}\psi_+,
		\]
		and the substitution in \eqref{fexpr1} yields
		\begin{equation}
			f=\Phi_{m,z}\varphi,\qquad
			\varphi:=\beta\Big(\half +C_{m,z}\beta\Big)^{-1}\psi_+.
			\label{fexpr2}
		\end{equation}
		
		(c) On the other hand, given $\psi_+$ let us define $f$ by \eqref{fexpr2}. By construction one has $f\in H^\half_\alpha(\RR^n\setminus\Sigma,\CC^N)$ with $(D_m-z)f=0$ in $\RR^n\setminus\Sigma$, therefore, $(D_m-z)f_+=0$ in $\Omega_+$. Further we have, with the help of Lemma~\ref{lem3},
		\begin{equation}
			\label{eq33}
			\begin{aligned}
				\cP_-\gamma_+ f_+&=\cP_-\gamma_+ \Phi_{m,z}\varphi=\cP_- \Big(-\half \cB\beta +C_{m,z}\Big)\varphi.
			\end{aligned}
		\end{equation}
		Using $\cP_-\cB=-\cP_-$ and the explicit expression of $\varphi$ we compute:
		\begin{align*}
			\cP_- &\Big(-\half \cB\beta +C_{m,z}\Big)\varphi=\cP_- \Big(-\half \cB +C_{m,z}\beta\Big)\beta\varphi=\cP_- \Big(\half  +C_{m,z}\beta\Big)\beta\varphi\\
			&=\cP_- \Big(\half  +C_{m,z}\beta\Big) \Big(\half +C_{m,z}\beta\Big)^{-1}\psi_+
			=\cP_- \psi_+ =\psi_+.
		\end{align*}
		The substitution into~\eqref{eq33} yields $\cP_-\gamma_+f_+=\psi_+$, and this shows the existence. The norm estimate \eqref{norms}
		follows from the boundedness of the mapping entering~\eqref{fexpr2}.
	\end{proof}
	
	In view of Lemma~\ref{lem14}, for $m\in\RR$ and $z\in\CC\setminus\spec A^\pm_m$ the maps
	\begin{gather*}
		E^\pm_m(z):\ \cP_\mp L^2(\Sigma,\CC^N)\to H^\half_\alpha(\Omega_\pm,\CC^N),\\
		E^\pm_m(z)\psi_\pm:= \text{ the solution $f_\pm$ to \eqref{bcprob}.}
	\end{gather*}
	are well-defined bounded linear operators. We further consider the bounded
	linear operators
	\[
	U^\pm_m(z):= \cP_\pm \gamma_\pm E^\pm_m(z):\ \cP_\mp L^2(\Sigma,\CC^N)\to \cP_\pm L^2(\Sigma,\CC^N),
	\]
	which are sometimes referred to as Poincar\'e-Steklov operators, see e.g.~\cite{bbz} for a detailed microlocal study in the case of smooth $\Sigma$. Let us establish some useful identities.

	\begin{lemma} For any $m,M\in\RR$ and $z\in\CC\setminus\RR$ one has
		\begin{align}
			r_+R_{m,M}(z) &= E_{m}^{+}(z)\cP_-\gamma R_{m,M}(z) + R_{m}^{+}(z)r_+, \label{interior resolvent perturbed op}\\
			r_-R_{m,M}(z) &= E_{M}^{-}(z)\cP_+\gamma R_{m,M}(z) + R_{M}^{-}(z)r_{-}, \label{exterior resolvent perturbed op}\\
			\cP_+ \gamma R_{m,M}(z) &= U_m^{+}(z)\cP_- \gamma R_{m,M}(z)+\cP_+ \gamma_+ R_m^{+}(z)r_{+}, \label{interior trace resolvent perturbed op}\\
			\cP_- \gamma R_{m,M}(z) &= U_{M}^{-}(z)\cP_+ \gamma R_{m,M}(z)+\cP_- \gamma_- R_{M}^{-}(z)r_{-}, \label{exterior trace resolvent perturbed op}
		\end{align}
		and
		\begin{multline}
			R_{m,M}(z)-e_+ R^+_m(z)r_+=e_+E^+_m(z)\cP_- \gamma R_{m,M}(z)\\
			+e_-E^-_M(z)\cP_+ \gamma R_{m,M}(z)+e_-R^-_M(z)r_-.
			\label{resident}
		\end{multline}	
	\end{lemma}
	
	\begin{proof}
		Let $g\in L^2(\RR^n,\CC^N)$ and $f:=R_{m,M}(z)g\in H^1(\RR^n,\CC^N)$.
		By definition, 
		\begin{gather*}
			\left\{\begin{array}{rl}
				(D_m-z)f_+ &= g_+ \text{ in }\Omega_+,\\
				\cP_- \gamma_+ f_+&=h_+ \text{ on } \Sigma,
			\end{array}\right.
			\qquad 
			\left\{\begin{array}{rl}
				(D_M-z)f_-&=g_- \text{ in }\Omega_-,\\
				\cP_+ \gamma_- f_-&=h_- \text{ on }\Sigma,
			\end{array}\right.\\
			h_\pm := \cP_\mp \gamma_\pm f_\pm \equiv \cP_\mp \gamma f,
		\end{gather*}
		which gives the representations
		\begin{align}
			f_+&=E^+_m(z)h_+ + R_m^+(z)g_+, \label{fplus}\\
			f_-&=E^-_M(z)h_- + R_M^-(z) g_-, \label{fminus}
		\end{align}
		and this gives the identities \eqref{interior resolvent perturbed op} and \eqref{exterior resolvent perturbed op}.
		By applying $\cP_+\gamma_+$ to both sides of \eqref{fplus} we obtain
		\[
		\cP_+\gamma_+ f_+\equiv \cP_+\gamma f = U^+_m(z)h_+ + \cP_+\gamma_+ R_m^+(z)g_+,
		\]
		which is precisely \eqref{interior trace resolvent perturbed op}. Similarly, by applying $\cP_-\gamma_-$ to \eqref{fminus}
		one arrives at \eqref{exterior trace resolvent perturbed op}.
		
		We further have
		\begin{align*}
			R_{m,M}(z)&g-e_+R^+_m(z)r_+ g=f- e_+R^+_m(z)g_+\\
			&\equiv e_+f_++e_-f_--e_+  R^+_m(z)g_+\equiv e_+\big(f_+-R^+_m(z)g_+\big)+e_- f_-\\
			&\stackrel{ \eqref{fplus}}{=} e_+E^+_m(z) h_+ +e_- f_-
			\stackrel{\eqref{fminus}}{=}e_+E^+_m(z)h_+ +e_-\big(E^-_M(z) h_- + R^-_M(z)g_-\big),
		\end{align*}
		and by recalling the definitions of $g_\pm$ and $h_\pm$ we arrive at~\eqref{resident}.	
	\end{proof}
	
	Let us pass to estimating various terms in the asymptotic regime $M\to+\infty$.
	
	\begin{lemma}\label{lem16est}
		Let $m\in\RR$ and $z\in\CC\setminus\RR$. Then one can find some $C>0$ and $M_0>0$ such that for all $M> M_0$ 
		one has the following estimates for $L^2$-based operator norms:
		\begin{align}
			\label{control, exterior MIT resolvent}
			\big\|R_M^-(z)\big\|& \le C M^{-1},\\
			\label{control, exterior solution MIT BVP}
			\big\|E^-_M(z)\big\|&\le C M^{-\half},\\
			\label{control, interior trace by exterior trace}
			\big\| \cP_+ \gamma R_{m,M}(z)\big\|&\le C\big\|\cP_- \gamma R_{m,M}(z)\big\|+C.
		\end{align}
		
	\end{lemma}
	
	\begin{proof}
		Throughout the proof we assume $M>0$. Due to Lemma~\ref{lem13spec}, the spectral theorem for self-adjoint operators implies
		\[
		\big\|R_M^-(z)\big\|=\dfrac{1}{\dist(z,\spec A^-_M)}\le \dfrac{1}{\dist\big(z,\RR\setminus(-M,M)\big)},
		\]
		which yields \eqref{control, exterior MIT resolvent}.
		
		Let $h\in \cP_+ L^2(\Sigma,\CC^N)$ and $f:=E_M^-(z)h$. Due to the definition of $E_M^-$ we have
		\begin{equation*}
			\begin{aligned}
				0&=\big\|(D_M-z)f\big\|^2_{L^2(\Omega_-,\CC^N)}\\
				&=\big\langle (D_0-z) f + M\beta  f, (D_0-z) f + M\beta f\big\rangle_{L^2(\Omega_-,\CC^N)}\\
				&=\underbrace{\big\|(D_0-z) f\big\|^2_{L^2(\Omega_-,\CC^N)}}_{\ge 0}+ M^2\|\beta f\big\|^2_{L^2(\Omega_-,\CC^N)}\\
				&\qquad
				+2M\Re \big[\langle D_0 f,\beta f\rangle_{L^2(\Omega_-,\CC^N)}\big]
				-2M\Re \big[z\langle \beta f,f\rangle_{L^2(\Omega_-,\CC^N)}\big]\\
				&\ge M^2\|f\|^2_{L^2(\Omega_-,\CC^N)}+2M\Re \big[\langle D_0 f,\beta f\rangle_{L^2(\Omega_-,\CC^N)}\big]
				-2M|z| \|f\|^2_{L^2(\Omega_-,\CC^N)},
			\end{aligned}
		\end{equation*}
		i.e.
		\begin{equation}
			\label{est000}
			\big(M-2|z|\big)\|f\|^2_{L^2(\Omega_-,\CC^N)}\le -2\Re \big[\langle D_0 f,\beta f\rangle_{L^2(\Omega_-,\CC^N)}\big].
		\end{equation}
		
		Using~ the partial integration \eqref{partint} we obtain
		\begin{equation}
			\label{redbeta}
			\begin{aligned}
				\langle D_0 f,\beta f\rangle_{L^2(\Omega_-,\CC^N)}
				&=\big\langle  \rmi (\alpha\cdot\nu) \gamma_- f,\beta \gamma_- f\big\rangle_{L^2(\Sigma,\CC^N)}
				+	\big\langle  f,D_0\beta f\big\rangle_{L^2(\Omega_-,\CC^N)}\\
				&=\big\langle  \rmi (\alpha\cdot\nu) \gamma_- f,\beta \gamma_- f\big\rangle_{L^2(\Sigma,\CC^N)}
				-	\langle \beta  f,D_0 f\rangle_{L^2(\Omega_-,\CC^N)},
			\end{aligned}
		\end{equation}
		therefore,
		\begin{align*}
			-2\Re &\langle D_0 f,\beta f\rangle_{L^2(\Omega_-,\CC^N)}
			=- \big\langle  \rmi (\alpha\cdot\nu) \gamma_- f,\beta \gamma_- f\big\rangle_{L^2(\Sigma,\CC^N)}\\
			&=\langle  \cB \gamma_- f, \gamma_- f\rangle_{L^2(\Sigma,\CC^N)}
			=\big\langle  (\cP_+-\cP_-) \gamma_- f, (\cP_+ +\cP_-)\gamma_- f\big\rangle_{L^2(\Sigma,\CC^N)}\\
			&=\big( \|\cP_+ \gamma_- f\|^2_{L^2(\Sigma,\CC^N)}-\|\cP_- \gamma_- f\|^2_{L^2(\Sigma,\CC^N)}\big)
			\le \|\cP_+ \gamma_- f\|^2_{L^2(\Sigma,\CC^N)}\equiv \|h\|^2_{L^2(\Sigma,\CC^N)}.
		\end{align*}
		then the substitution into \eqref{est000} yields, for $M> 2|z|$,
		\[
		\|E_M^-(z) h\|^2_{L^2(\Omega_-,\CC^N)}\equiv
		\|f\|^2_{L^2(\Omega_-,\CC^N)}\le\dfrac{1}{M-2|z|}\|h\|^2_{L^2(\Sigma,\CC^N)},
		\]
		which in turn implies \eqref{control, exterior solution MIT BVP}.
		
		Due to \eqref{interior trace resolvent perturbed op} we have
		\begin{align*}
			\big\|\cP_+ \gamma R_{m,M}(z)\big\| &\le \big\|U_m^{+}(z)\cP_- \gamma R_{m,M}(z)\big\|+\big\|\cP_+ \gamma R_m^{+}(z)r_{+}\big\|\\
			&\le \big\|U_m^{+}(z)\big\|\cdot\big\|\cP_- \gamma R_{m,M}(z)\big\| + \big\|\cP_+ \gamma R_m^{+}(z)r_{+}\big\|,
		\end{align*}
		which proves the remaining estimate \eqref{control, interior trace by exterior trace}.
	\end{proof}
	
	In order to arrive at the resolvent convergence, we now recall the assumption \eqref{robinineq}, which
	we rewrite as follows:
	\begin{equation}
		\label{assump}
		\begin{minipage}{110mm}
			There exists a function $\rho:[0,\infty)\to [0,\infty)$ with
			\[
			\lim_{M\to+\infty}\dfrac{\rho(M)}{M^2}=0
			\]
			such that for any $f\in H^1(\Omega_-)$ and $M>0$ it holds that
			\[
			\int_{\Omega_-} |\nabla f|^2\dd x-M\int_{\partial\Omega_-} |f|^2\dd S\ge -\big(M^2+\rho(M)\big)\int_{\Omega_-}|f|^2\dd x.
			\]
		\end{minipage}
	\end{equation}	
	To have simpler expressions we begin with the massless case $m=0$.
	\begin{lemma}\label{lem17conv}
		Let \eqref{assump} be satisfied. Then for any $z\in\CC\setminus\RR$ one has
		\[
		\big\|R_{0,M}(z)-R^+_0(z)\oplus \zero\big\|=O\Big( \dfrac{\sqrt{\rho(M)}+\sqrt{M}}{M}\Big),\quad M\to +\infty.
		\]
	\end{lemma}
	
	\begin{proof}
		Below we use the notation $a_M\lesssim b_M$ to say that there is $C > 0$, independent of $M$, such that $a_M \leq Cb_M$. All the subsequent inequalities are valid for $M$ large enough.
		
		As $R^+_0(z)\oplus \zero=e_+R^+_0(z)r_+$, by \eqref{resident} we have
		\begin{gather*}
			R_{0,M}(z)-R^+_0(z)\oplus \zero=J_1+J_2+J_3,\\
			J_1:= e_+E^+_0(z)\cP_- \gamma R_{0,M}(z),\quad
			J_2:=e_-E^-_M(z)\cP_+ \gamma R_{0,M}(z),\quad
			J_3:=e_-R^-_M(z)r_-,
		\end{gather*}	
		and it is sufficient to estimate the norm convergence rate to zero for each  $J_j$.
		Note that $\|J_3\|\lesssim M^{-1}$ due to~\eqref{control, exterior MIT resolvent}. Further, in view of \eqref{control, exterior solution MIT BVP} and \eqref{control, interior trace by exterior trace},
		\begin{equation}
			\label{j1j2}
			\begin{aligned}
				\|J_1\|&\lesssim \big\| \cP_- \gamma R_{0,M}(z)\big\|,\\	
				\|J_2\|&\le \big\|E^-_M(z)\big\| \cdot \big\|\cP_+ \gamma R_{0,M}(z)\big\|\\
				&\lesssim M^{-\half} \big\|\cP_+ \gamma R_{0,M}(z)\big\|\lesssim M^{-\half}\big( \big\|\cP_- \gamma R_{0,M}(z)\big\|+1\big),
			\end{aligned}
		\end{equation}
		therefore, one needs a bound for $\big\|\cP_- \gamma R_{0,M}(z)\big\|$.
		
		For each $f\in H^1(\RR^n,\CC^N)$ one has $\gamma_\pm f=\gamma f$.
		Recall that
		\begin{align*}
			\|D_{M}f&\|^2_{L^2(\Omega_-,\CC^N)}=\|D_0f+M\beta f\|^2_{L^2(\Omega_-,\CC^N)}\\
			&=\|D_0f\|^2_{L^2(\Omega_-,\CC^N)}+M^2\|f\|^2_{L^2(\Omega_-,\CC^N)}
			+2M\Re \langle D_0 f,\beta f\rangle_{L^2(\Omega_-,\CC^N)}\\
			&\stackrel{\eqref{redbeta}}{=}
			\|D_0f\|^2_{L^2(\Omega_-,\CC^N)}+M^2\|f\|^2_{L^2(\Omega_-,\CC^N)}+
			M\big\langle \rmi(\alpha\cdot\nu)\gamma f,\beta \gamma f\big\rangle_{L^2(\Sigma,\CC^N)}\\
			&\equiv \|D_0f\|^2_{L^2(\Omega_-,\CC^N)}+M^2\|f\|^2_{L^2(\Omega_-,\CC^N)}-M\langle \cB\gamma f,\gamma f\rangle_{L^2(\Sigma,\CC^N)}\\
			&\equiv \|D_0f\|^2_{L^2(\Omega_-,\CC^N)}+M^2\|f\|^2_{L^2(\Omega_-,\CC^N)}\\
			&\qquad -M\big( \|\cP_+ \gamma f\|^2_{L^2(\Sigma,\CC^N)}
			-\|\cP_- \gamma f\|^2_{L^2(\Sigma,\CC^N)}\big),
		\end{align*}
		hence,
		\begin{align}\label{Qestimate}
			\begin{split}
				\|A_{0,M}f&\|^2_{L^2(\RR^n,\CC^N)}=\|D_{0}f\|^2_{L^2(\Omega_+,\CC^N)}+\|D_{M}f\|^2_{L^2(\Omega_-,\CC^N)}\\
				&=\|D_0f\|^2_{L^2(\RR^n,\CC^N)}+M^2\|f\|^2_{L^2(\Omega_-,\CC^N)}\\
				&\qquad+M\|\cP_- \gamma f\|^2_{L^2(\Sigma,\CC^N)}-M\|\cP_+ \gamma f\|^2_{L^2(\Sigma,\CC^N)}\\
				&=\|\nabla f\|^2_{L^2(\RR^n,\CC^N)}+M^2\|f\|^2_{L^2(\Omega_-,\CC^N)}\\
				&\qquad+M\|\cP_- \gamma f\|^2_{L^2(\Sigma,\CC^N)}-M\|\cP_+ \gamma f\|^2_{L^2(\Sigma,\CC^N)}\\
				&\ge \Big( \|\nabla f\|^2_{L^2(\Omega_-,\CC^N)}-M\|\gamma f\|^2_{L^2(\Sigma,\CC^N)} +M^2\|f\|^2_{L^2(\Omega_-,\CC^N)}\Big)\\
				&\qquad +M\|\cP_- \gamma f\|^2_{L^2(\Sigma,\CC^N)}\\
				&\stackrel{\eqref{assump}}{\ge}-\rho(M)\|f\|^2_{L^2(\Omega_-,\CC^N)}+M\|\cP_- \gamma f\|^2_{L^2(\Sigma,\CC^N)},
			\end{split}
		\end{align}
		and by substituting $f:=R_{0,M}(z)g$ with any $g\in L^2(\RR^n,\CC^N)$ we arrive at
		\begin{align*}
			\big\|\cP_- \gamma R_{0,M}(z)g\big\|^2_{L^2(\Sigma,\CC^N)}&\le \dfrac{1}{M}\|g+zR_{0,M}(z)g\|^2_{L^2(\RR^n,\CC^N)}\\
			&\qquad +\dfrac{\rho(M)}{M}\big\|R_{0,M}(z)g\big\|^2_{L^2(\Omega_-,\CC^N)}\\
			&\le \dfrac{\big\|1+zR_{0,M}(z)\big\|^2}{M}\|g\|^2_{L^2(\RR^n,\CC^N)}\\
			&\qquad +\dfrac{\rho(M)}{M}\big\|R_{0,M}(z)g\big\|^2_{L^2(\Omega_-,\CC^N)}.
		\end{align*}
		According to the spectral theorem for self-adjoint operators, the uniform estimate $\big\|R_{0,M}(z)\big\|\leq |\Im z|^{-1}$ holds, and the above inequality reads as
		\begin{equation}
			\label{lokal111}
			\big\|\cP_- \gamma R_{0,M}(z)\big\|\lesssim\dfrac{1}{\sqrt{M}} +\sqrt{\dfrac{\rho(M)}{M}}\|r_-R_{0,M}(z)\|.
		\end{equation}	
		With the help of~\eqref{exterior resolvent perturbed op}, we estimate
		\begin{align*}
			\big\|r_-R_{0,M}(z)\big\|&\le \big\|E_{M}^{-}(z)\big\|\cdot \big\|\cP_+\gamma R_{0,M}(z)\big\| + \big\|R_{M}^{-}(z)\big\|\\
			\text{(use Lemma~\ref{lem16est}) }	&\lesssim \dfrac{1}{\sqrt{M}}\big\|\cP_+\gamma R_{0,M}(z)\big\| +\dfrac{1}{M},
		\end{align*}	
		and the substitution into \eqref{lokal111} yields
		\begin{equation}
			\label{lokal222}
			\begin{aligned}
				\big\|\cP_- \gamma R_{0,M}(z)\big\|&\lesssim \dfrac{1}{\sqrt{M}} + \dfrac{1}{\sqrt{M}}\sqrt{\dfrac{\rho(M)}{M^2}} +\sqrt{\dfrac{\rho(M)}{M^2}}\big\|\cP_+\gamma R_{0,M}(z)\big\|\\
				&\lesssim \dfrac{1}{\sqrt{M}} +\sqrt{\dfrac{\rho(M)}{M^2}}\big\|\cP_+\gamma R_{0,M}(z)\big\|.
			\end{aligned}
		\end{equation}	
		Hence, by \eqref{control, interior trace by exterior trace} we can estimate
		\[
		\big\| \cP_+ \gamma R_{0,M}(z)\big\|\lesssim \big\|\cP_- \gamma R_{0,M}(z)\big\|+1
		\lesssim\sqrt{\dfrac{\rho(M)}{M^2}} \big\|\cP_+\gamma R_{0,M}(z)\big\|+1,
		\]
		which results, for some $c>0$, in
		\[
		\Big( 1-c\sqrt{\dfrac{\rho(M)}{M^2}}\ \Big)\big\|\cP_+\gamma R_{0,M}(z)\big\|\lesssim 1, \text{ i.e }
		\big\|\cP_+\gamma R_{0,M}(z)\big\|\lesssim 1,
		\]
		and \eqref{lokal222} yields
		\[
		\|\cP_- \gamma R_{0,M}\|\lesssim \dfrac{1}{\sqrt{M}} +\sqrt{\dfrac{\rho(M)}{M^2}}.
		\]
		The substitution into \eqref{j1j2}
		gives
		\[
		\|J_1\|\lesssim\dfrac{1}{\sqrt{M}} +\sqrt{\dfrac{\rho(M)}{M^2}},
		\quad
		\|J_2\|\lesssim \dfrac{1}{\sqrt{M}},
		\]
		and, finally,
		\begin{align*}
			\big\|R_{0,M}(z)-R^+_0(z)\oplus \zero\big\|&\le \|J_1\|+\|J_2\|+\|J_3\|\\
			&\lesssim\dfrac{1}{\sqrt{M}} +\sqrt{\dfrac{\rho(M)}{M^2}} + \dfrac{1}{\sqrt{M}}+\dfrac{1}{M}\\
			&\lesssim\dfrac{1}{\sqrt{M}} +\sqrt{\dfrac{\rho(M)}{M^2}}. \qedhere
		\end{align*}
	\end{proof}
	
	The following assertion will allow us to cover arbitrary masses $m$ and indicate a path for further extensions.

	\begin{lemma}\label{lem18rate}
		Let $\cH$ be a Hilbert space and $\cH_0\subset\cH$ be a closed subspace. Let $T_n$ be self-adjoint operators in $\cH$
		and $T$ be a self-adjoint operator in $\cH_0$. Further let $V\in \cL(\cH)$ be self-adjoint with $V(\cH_0)\subset\cH_0$.
		If for some $z\in\CC\setminus\RR$ one has
		\begin{equation*}
			a_n:=\big\|(T_n-z)^{-1}-(T-z)^{-1}\oplus \zero\big\|\xrightarrow{n\to\infty}0,
		\end{equation*}
		with respect to the decomposition $\cH=\cH_0\oplus \cH_0^\perp$, then for some $C>0$ one has
		\[
		\big\|(T_n+V-z)^{-1}-(T+V-z)^{-1}\oplus \zero\big\|\le Ca_n\xrightarrow{n\to\infty}0.
		\]	
	\end{lemma}
	
	\begin{proof}
		We have $(T_n-z)^{-1}-(T_n+V-z)^{-1}=(T_n-z)^{-1}V(T_n+V-z)^{-1}$,
		which can be rewritten as
		\[
		(T_n-z)^{-1}=K_n(T_n+V-z)^{-1},\quad K_n:=I+(T_n-z)^{-1}V,
		\]
		and then $K_n:\cH\to\cH$ is an isomorphism,
		with
		$(T_n+V-z)^{-1}=K_n^{-1}(T_n-z)^{-1}$.
		
		Analogously, $(T-z)^{-1}=K(T+V-z)^{-1}$, where $K:=I+(T-z)^{-1}V:\cH_0\to\cH_0$
		is an isomorphism, with $(T+V-z)^{-1}=K^{-1}(T-z)^{-1}$.
		
		Note that by assumption
		\begin{align*}
			K_n&\xrightarrow{n\to\infty} I+\big((T-z)^{-1}\oplus \zero\big)V\equiv K\oplus I,\\
			\|K_n- K\oplus I\|&=\Big\|\big((T_n-z)^{-1}-(T-z)^{-1}\oplus \zero\big)V\Big\|\\
			&\le \big\|(T_n-z)^{-1}-(T-z)^{-1}\oplus \zero\big\|\cdot\|V\big\|=\|V\| a_n=C_1a_n,
		\end{align*}
		and the limit operator $K\oplus I$ has a bounded inverse. Hence,
		\begin{align*}
			K_n^{-1}&\xrightarrow{n\to\infty}(K\oplus I)^{-1},\\
			\|K_n^{-1}-(K\oplus I)^{-1}\|&=
			\Big\| K_n^{-1}\big((K\oplus I)-K_n\big)(K\oplus I)^{-1}\|\\
			&\le \|K_n^{-1}\|\cdot \big\|(K\oplus I)^{-1}\big\|
			\cdot\big\|(K\oplus I)-K_n\big\|\le C_2 a_n,
		\end{align*}
		and then
		\begin{align*}
			(T_n+V-z)^{-1}&=K_n^{-1}(T_n-z)^{-1}\xrightarrow{n\to\infty} (K\oplus I)^{-1}\big((T-z)^{-1}\oplus \zero\big)\\
			&\qquad\equiv (T+V-z)^{-1}\oplus \zero,\\
			\big\|(T_n+V-z)^{-1}&-(T+V-z)^{-1}\oplus \zero\big\|\\
			&=\Big\| K_n^{-1}(T_n-z)^{-1}- (K\oplus I)^{-1}\big((T-z)^{-1}\oplus \zero\big)\Big\|\\
			&\le \Big\| \big(K_n^{-1}-(K\oplus I)^{-1}\big)(T_n-z)^{-1}\Big\|\\
			&\quad 	+ \Big\| (K\oplus I)^{-1} \big( (T_n-z)^{-1} -(T-z)^{-1}\oplus \zero\big)\Big\|\\
			&\le \big\|K_n^{-1}-(K\oplus I)^{-1}\big\|\cdot\big\|(T_n-z)^{-1}\big\|\\
			&\quad + \big\| (K\oplus I)^{-1}\big\|\cdot \big\|(T_n-z)^{-1} -(T-z)^{-1}\oplus \zero\big\|\\
			&\le C_3 a_n. \qedhere
		\end{align*}	
	\end{proof}
	
	\begin{theorem}\label{thm19}
		Let \eqref{assump} be satisfied. Then for any $m\in\RR$ and any $z\in\CC\setminus\RR$ one has
		\[
		\big\|R_{m,M}(z)-R^+_m(z)\oplus \zero\big\|=O\Big( \dfrac{\sqrt{\rho(M)}+\sqrt{M}}{M}\Big),\quad M\to +\infty.
		\]
	\end{theorem}	
	
	\begin{proof}
		The case $m=0$ is already covered by Lemma~\ref{lem17conv}. The case $m\ne 0$
		is included with the help of Lemma~\ref{lem18rate} for $T_n:=A_{0,M}$, $T:=A^+_0$, $V:=m\one_{\Omega_+}\beta$
		by noting that $L^2(\Omega_+,\CC^N)\oplus 0$ is an invariant subspace for $V$.
	\end{proof}
	
	\begin{remark}
		Actually, Lemma~\ref{lem18rate} allows for inclusion of other perturbation types, like additional matrix potentials. The boundedness assumption on the perturbation can certainly be relaxed, see e.g. a version in~\cite{barb} for $n=2$. However, we are not pursuing this direction here.
	\end{remark}
	

	\section{Admissible domains and convergence rates}\label{sec5}
	
	Recall that for an open set $U\subset\RR^n$ with compact Lipschitz boundary and outer unit normal $\nu_U$ and $M\in\RR$, the quantity
	\[
	\Lambda(U,M):=\inf_{f\in H^1(U),\, f\not\equiv 0}\dfrac{\dint_U |\nabla f|^2\dd x-M\dint_{\partial U} |f|^2\dd S}{\dint_U |f|^2\dd x}
	\]
	is the bottom of the spectrum of the Laplacian in $U$ with the  Robin boundary condition $\partial_{\nu_U} f=Mf$ on $\partial U$, and the central convergence assumption \eqref{assump} is equivalently rewritten as
	\begin{equation}
		\label{assump2}
		\Lambda(\Omega_-,M)\ge -M^2-\rho(M),\quad \rho(M)=o(M^{2}),\quad M\to+\infty.
	\end{equation}
	
	The study of the asymptotic behavior of $\Lambda(U,M)$ for $M\to+\infty$ has been an active topic in recent decades, see the reviews in \cite{bfk,kobp}: in general, one has
	$\Lambda(U,M)\ge -cM^2$ for some $c>0$ and large $M$, however, the smallest admissible $c$ can be arbitrarily large, which happens e.g. if $U$ is a planar domain having a sufficiently acute corner \cite{kobp}. In other words, the inequality \eqref{assump2} clearly fails for general $\Omega_\pm$.
	The aim of the present section is to find geometric properties of $\Omega_\pm$ for which the validity of \eqref{assump2} can still be guaranteed.
	
	First note that the following result was shown in \cite[Thm.~1.1]{kovp}:
	
	\begin{lemma}\label{lem-kovp}
		If an open set $U\subset\RR^n$ has a compact $C^{1,1}$-smooth boundary, then $\Lambda(U,M)=-M^2+O(M)$ for large $M$.
	\end{lemma} 
	
	 This leads to the following convergence result:
	
	\begin{theorem}\label{thmc11}
		Let $\Omega_+$ be a domain with compact $C^{1,1}$ boundary, then for any $m\in\RR$ and any $z\in\CC\setminus\RR$ one has
		\[
		\big\|R_{m,M}(z)-R^+_m(z)\oplus \zero\big\|=O\Big(\dfrac{1}{\sqrt{M}}\Big),\quad M\to +\infty.
		\]
	\end{theorem}

	\begin{proof}
		The preceding discussion says that \eqref{assump2} holds for $\rho(M)=cM$ with a sufficiently large $c>0$, and the substitution into~Theorem~\ref{thm19} gives the conclusion.
	\end{proof}

	In order to cover a class of non-smooth domains we start with the following observation:
	
	\begin{lemma}\label{lem22convex}
		Let $U\subset\RR^n$ be a convex domain and $U^\cc:=\RR^n\setminus \overline{U}$, then for all $M>0$ one has $\Lambda(U^\cc,M)\ge -M^2$.
	\end{lemma}
	
	\begin{proof}
		Note that the sought conclusion is known under the additional assumption that $U$ has $C^2$-smooth
		boundary~\cite[Thm.~2]{Pankrashkin_2016}. We are going to extend it to the general case using an approximation argument.

		First recall that the boundary integral in the definition of $\Lambda(U^\cc,M)$ makes sense at least for compactly supported functions, as convex domains have locally Lipschitz boundaries~\cite[Corol.~1.2.2.3]{Grisvard_book}. Then it follows that the restrictions of $C^\infty_c(\RR^n)$-functions on $U^\cc$ are dense in $H^1(U^\cc)$, see \cite[Thm.~1.4.2.1]{Grisvard_book}. Hence, it is sufficient to show that
		\begin{equation*}
			M\int_{\partial U}|f|^2\dd S\le \int_{U^\cc} |\nabla f|^2\dd x+M^2\int_{U^\cc}|f|^2\dd x,
			\quad f\in C^\infty_c(\RR^n),\ M>0.
		\end{equation*}

		Let $f\in C^\infty_c(\RR^n)$ and $M>0$. Let $B$ be a large open ball containing $\supp f$, then $W:=U\cap B$ is a bounded convex domain, and for $W^\cc:=\RR^n\setminus \overline{W}$ we have
		\begin{align*}
			\int_{U^\cc} |\nabla f|^2\dd x&=\int_{W^\cc} |\nabla f|^2\dd x,\\
			\int_{U^\cc} |f|^2\dd x&=\int_{W^\cc} |f|^2\dd x,\\
			\int_{\partial U}|f|^2\dd S&=\int_{\partial W}|f|^2\dd S.
		\end{align*}
		Hence, it is sufficient to show that
		\begin{equation}
			\label{inequp01}
			M\int_{\partial W}|f|^2\dd S\le \int_{W^\cc} |\nabla f|^2\dd x+M^2\int_{W^\cc}|f|^2\dd x.
		\end{equation}

		A standard result of convex analysis, see e.g. \cite[Lemma 3.2.3.2]{Grisvard_book}, says that there exist bounded convex $C^2$-smooth domains $W_p\subset\RR^n$, $p\in\NN$, such that $W\subset W_p$ for all $p$ with $d_H(\partial W_p,\partial W)\xrightarrow{p\to\infty}0$, where $d_H$ stands for the Hausdorff distance, and one can find open subsets $V_k\subset\RR^n$, $k\in\{1,\dots,K\}$,
		with the following properties\textup{:}
		\begin{itemize}
			\item[(a)] for each $k$ there exist Cartesian coordinates $y^k_1,\dots,y^k_n$ in which $V_k$ is a hypercube,
			\[
			V_k=\Big\{(y^k_1,\dots,y^k_n): -a^k_j<y^k_j<a^k_j \text{ for all } j\in\{1,\dots,n\}\Big\},\quad a^k_j>0,
			\]
			\item[(b)] for all $p\in\NN$ and each $k\in\{1,\dots,K\}$ there exist a Lipschitz function $h^k$ and a $C^2$-smooth function $h^k_p$ defined on
			\[
			V'_k=\Big\{(y^k_1,\dots,y^k_{n-1}): -a^k_j<y^k_j<a^k_j \text{ for all } j\in\{1,\dots,n-1\}\Big\}
			\]
			such that
			\begin{align*}
				& h^k(z^k)\le h^k_p(z^k),\ 			
				\big|h^k(z^k)\big|\le \frac{a^k_n}{2},\ \big|h^k_p(z^k)\big|\le \frac{a^k_n}{2} \text{ for all }z^k\in V'_k,\\
				W\cap V_k&=\big\{ y^k=(z^k,y^k_n)\in V_k:\, z^k\in V'_k,\ y^k_n< h^k(z^k)\big\}, \\
				W_p\cap V_k&=\big\{ y^k=(z^k,y^k_n)\in V_k:\, z^k\in V'_k,\ y^k_n< h^k_p(z^k)\big\},\\
				\partial W\cap V_k&=\big\{ y^k=(z^k,y^k_n)\in V_k:\, z^k\in V'_k,\ y^k_n= h^k(z^k)\big\},\\
				\partial W_p\cap V_k&=\big\{ y^k=(z^k,y^k_n)\in V_k:\, z^k\in V'_k,\ y^k_n= h^k_p(z^k)\big\},
			\end{align*}
			\item[(c)] $\partial W\subset\bigcup_{k=1}^K V_k$ and $\partial W_p\subset\bigcup_{k=1}^K V_k$ for all $p\in\NN$,
			\item[(d)] for each $k\in\{1,\dots,K\}$ one has $h^k_p\xrightarrow{p\to\infty} h^k$ uniformly in $V'_k$,
			\item[(e)] for each $k\in\{1,\dots,K\}$ one has $\nabla h^k_p(z^k)\xrightarrow{p\to\infty} \nabla h^k(z^k)$ for a.e. $z^k\in V'_k$,
			\item[(f)] there is an $L>0$ such that $	\big|\nabla h^k(z^k)\big|\le L$ and $\big|\nabla h^k_p(z^k)\big|\le L$
			for a.e. $z^k\in V'_k$, all $k\in\{1,\dots,K\}$ and all $p\in\NN$.
		\end{itemize}
		
		As noted above, for each $p$ and $W_p^\cc:=\RR^n\setminus\overline{W_p}$ one has $\Lambda(W_p^\cc,M)\ge -M^2$, in particular,
		\begin{equation}
			\label{inequp}
			M\int_{\partial W_p}|f|^2\dd S\le \int_{W_p^\cc} |\nabla f|^2\dd x+M^2\int_{W_p^\cc}|f|^2\dd x,
		\end{equation}	
		and we would like to pass to the limit $p\to\infty$ in order to obtain~\eqref{inequp01}.
		
		By $d_H(\partial W_p,\partial W)\xrightarrow{p\to\infty}0$ one has
		$W^\cc_p\subset W^\cc$ and $|W^\cc\setminus W_p^\cc|\xrightarrow{p\to \infty}0$,
		which yields
		\begin{equation}
			\label{intp}
			\int_{W_p^\cc} |\nabla f|^2\dd x+M^2\int_{W_p^\cc}|f|^2\dd x\xrightarrow{p\to\infty}\int_{W^\cc} |\nabla f|^2\dd x+M^2\int_{W^\cc}|f|^2\dd x.
		\end{equation}
		Let us show that
		\begin{equation}
			\label{upf}
			\int_{\partial W_p} |f|^2\dd S\xrightarrow{p\to\infty}\int_{\partial W} |f|^2\dd S.
		\end{equation}
		Let $\chi_1,\dots,\chi_K\in C^\infty_c(\RR^n)$ with $\supp\chi_k\subset V_k$ and $\chi_1+\dots+\chi_K=1$
		on a neighborhood of $\partial W$, then one also has $\chi_1+\dots+\chi_K=1$ on $\partial W_p$ for all sufficiently large $p$, and due to (c) it holds that
		\begin{equation*}
			\begin{array}{lll}
				\displaystyle\int_{\partial W} |f|^2\dd S=\sum_{k=1}^K J^k, & &\displaystyle\int_{\partial W_p} |f|^2\dd S=\sum_{k=1}^K J^k_p,\\[4mm]
				J^k:=\displaystyle\int_{V_k\cap \partial W} \chi_k|f|^2\dd S,& & J^k_p:=\displaystyle\int_{V_k\cap \partial W_p} \chi_k|f|^2\dd S.
			\end{array}
		\end{equation*}
		In the coordinate form we have
		\begin{align*}
			J^k&=\int_{V'_k} \chi_k\big(z^k,h^k(z^k)\big) \Big| f\big(z^k,h^k(z^k)\big)\Big|^2 \sqrt{1+|\nabla h^k(z^k)|^2} \,\dd z^k,\\
			J^k_p&=\int_{V'_k} \chi_k\big(z^k,h^k_p(z^k) \big)\Big| f(z^k,h^k_p(z^k)\big)\Big|^2 \sqrt{1+|\nabla h^k_p(z^k)|^2}\, \dd z^k.
		\end{align*}
		The subintegral expression in $J_p^k$ converges for a.e. $z^k\in V'_k$ to that in $J^k$ due to (d) and (e), and it is
		uniformly controlled by a suitably large constant due to (f). Therefore, $J_p^k\xrightarrow{p\to\infty}J^k$
		by the dominated convergence theorem, which gives \eqref{upf}. By using \eqref{intp} and \eqref{upf}
		we can pass to the limit $p\to\infty$ in \eqref{inequp} to arrive at the sought estimate \eqref{inequp01}.
	\end{proof}

	As a direct consequence we obtain:
	\begin{theorem}\label{thmconvex}
		Let $\Omega_+$ be a convex domain with compact boundary,  then for any $m\in\RR$ and any $z\in\CC\setminus\RR$ one has
		\[
		\big\|R_{m,M}(z)-R^+_m(z)\oplus \zero\big\|=O\Big(\dfrac{1}{\sqrt{M}}\Big),\quad M\to +\infty.
		\]
	\end{theorem}
	
	\begin{proof}
		Due to Lemma~\ref{lem22convex}, the assumption \eqref{assump2} holds with $\rho(M)=0$, and the conclusion follows by Theorem~\ref{thm19}.	
	\end{proof}

	\begin{remark}
		The recent paper~\cite{pankrashkin2025mitbagnonsmoothconvex} showed the convergence of the individual eigenvalues of $A_{m,M}$ to those of $A_m^+$ for the case when $\Omega_+$ is a bounded convex domain having a piecewise boundary and bounded mean curvature. The preceding theorem upgrades the eigenvalue convergence to the norm-resolvent convergence and removes the piecewise smoothness assumption.
	\end{remark}

	In order to allow for a more general class of domains let us introduce a definition:
	\begin{definition}\label{deff}
		An open set $U\subset\RR^n$ will be called \emph{locally $C^1$--convexifiable near a point $z\in \partial U$},
		if one can find
		\begin{itemize}
			\item an open neighborhood $U_{z}\subset\RR^n$ of $z$, 
			\item a convex domain $W^{(z)}$ with $z\in \partial W^{(z)}$,
			\item an open neighborhood $W_{z}\subset\RR^n$ of $z$,
			\item a $C^1$-diffeomorphism $\Phi_z: W_{z}\to U_{z}$ with $\Phi_z(z)=z$ and  $\Phi_z(W_z\cap W^{(z)})=U_z\cap U$.
		\end{itemize}
		If the above properties hold for each $z\in \partial U$, then $U$ is called \emph{locally $C^1$--convexifiable}.
	\end{definition}
	
	\begin{remark}\label{rmk00}
		In Definition~\ref{deff}, one can additionally ask for $D\Phi_z(z)=I$, which entails no loss of generality. In fact, if $A:=D\Phi_z(z)$, then consider the map
		\[
		L:\ x\mapsto A(x-z)+z,
		\]
		and denote
		\[
		\Tilde W^{(z)}:=L(W^{(z)}),\quad \Tilde W_z:=L(W_z),\quad \Tilde \Phi_z:=\Phi_z\circ L^{-1}.
		\]
		Note that $\Tilde W^{(z)}$ is convex as the image of a convex set under an affine map. Further,
		\[
		\Tilde \Phi_z(\Tilde W_z\cap \Tilde W^{(z)})=\Phi_z\Big( L^{-1} \big(L(W_z\cap W^{(z)})\big)\Big)=\Phi_z(W_z\cap W^{(z)})=U_z\cap U,
		\]
		therefore,
		all assumptions are satisfied if one replaces $W^{(z)}$, $W_z$ and $\Phi_z$ by $\Tilde W^{(z)}$, $\Tilde W_z$ and $\Tilde \Phi_z$, respectively. It remains to check that
		\[
		D\Tilde \Phi_z(z)=D\Phi_z \big(L^{-1}(z)\big) DL^{-1}(z)=D\Phi_z(z) A^{-1}=A A^{-1}=I.
		\]
	\end{remark}
	
	By definition, each convex open set is locally $C^1$--convexifiable (one simply takes $W^{(z)}=U$ for all $z$). Furthermore, the image of a convex domain under diffeomorphisms of $\RR^n$ is locally $C^1$--convexifiable as well. Another example is given with the following observation:
	\begin{lemma}\label{lemC1convexifiable}
		If $U$ coincides with a $C^1$-domain near $z\in\partial U$, then $U$ is locally $C^1$--convexifiable near $z$.
	\end{lemma}	
	
	\begin{proof}
		Without loss of generality assume $z=0$. Let $y_1,\dots,y_n$ be Cartesian coordinates centered
		at $0$ such that for some neighborhood $U_0$ of $0$ the set $U\cap U_0$ coincides with $U_0\cap\big\{x: x_n< h(x_1,\dots,x_{n-1})\big\}$, where $h$ is a $C^1$ function with $h(0)=0$. Then the map
		\[
		\Psi:\quad (x_1,\dots,x_n)\mapsto \big(x_1,\dots,x_{n-1},x_n-h(x_1,\dots,x_{n-1})\big)
		\]
		is easily shown to be a diffeomorphism mapping $U_0\cap U$ onto $W_0\cap W^{(0)}$, where $W_0$ is a neighborhood of $0$ and $W^{(0)}:=\RR^{n-1}\times(-\infty,0)$ is convex, and Definition~\ref{deff} is satisfied by taking $\Phi_0:=\Psi^{-1}$.
	\end{proof}
	
	As a result, any open set $U\subset\RR^n$ with $C^1$-smooth boundary is locally $C^1$--convexifiable. By combining the above statements one easily shows that planar domains whose boundary is $C^1$-piecewise smooth without concave corners are locally $C^1$--convexifiable too. Let us pass to spectral consequences.

	\begin{lemma}\label{lemmpseudo}
		Let an open set $U\subset\RR^n$ be locally $C^1$--convexifiable with compact boundary. Then
		for $U^\cc:=\RR^n\setminus\overline{U}$ one has $\Lambda(U^\cc,M)\ge -\big(1+o(1)\big)M^2$
		for $M\to +\infty$.
	\end{lemma}
	
	\begin{proof}
		Let us pick some $\eps\in(0,1)$. For each $z\in\partial U$ choose $U_z$, $W_z$, $W^{(z)}$, $\Phi_z$
		as in Definition~\ref{deff}. As noted in Remark~\ref{rmk00}, one can additionally assume $D\Phi_z(z)=I$, and by reducing the size of the neighborhoods we may assume additionally that $\|D\Phi_z(x)-I\|<\eps$
		for all $x\in W_z$ and all $z$. The sets $U_z$ form an open covering of the compact set $\partial U$, hence, one can find $K\in\NN$ and $z_1,\dots, z_K\in \partial U$ such that
		\[
		\partial U \subset\bigcup_{k=1}^K U_{z_k}.
		\]
		
		Let $\chi_1,\dots,\chi_K$ with $\chi_k\in C^\infty_c(U_{z_k})$, $0\le\chi_k\le 1$ and $\chi_1+\ldots+\chi_K=1$  on a neighborhood of $\partial U$.
		Denote $\chi_0:=1-(\chi_1+\dots+\chi_K)$ and set
		\[
		\psi:=\sum_{k=0}^K \chi_k^2,\quad \psi_k:=\dfrac{\chi_k}{\sqrt{\psi}},
		\]
		then $\psi_k\in C^\infty_c(U_{z_k})$ for $k\in \{1,\dots,K\}$ and $\psi_0\in C^\infty(U^\cc)$ with $\psi_0=0$ on a neighborhood of $\partial U$, with $\psi_0^2+\dots+\psi_K^2=1$ on $\overline{U^\cc}$.
		
		Let $f\in H^1(U^\cc)$ and set $h:=\sum_{k=0}^K |\nabla \psi_k|^2$. Then for $f_k:=\psi_k f$ one has
		\begin{align}
			\int_{U^\cc} |\nabla f|^2\dd x&=\int_{U^\cc} |\nabla f_0|^2\dd x + \sum_{k=1}^K\int_{U_{z_k}\cap U^\cc} |\nabla f_k|^2\dd x -\int_{U^\cc} h |f|^2\dd x, \label{ff1}\\
			\int_{U^\cc} |f|^2\dd x&= \int_{U^\cc} |f_0|^2\dd x + \sum_{k=1}^K\int_{U_{z_k}\cap U^\cc} |f_k|^2\dd x , \label{ff2}\\
			\int_{\partial U}|f|^2\dd S&=\sum_{k=1}^K\int_{U_{z_k}\cap \partial U}|f_k|^2\dd S. \label{ff3}
		\end{align}
		
		Recall that for each $k\in\{1,\dots,K\}$ we have $U_{z_k}\cap U^\cc=\Phi_{z_k}(W_{z_k}\cap (W^{(z_k)})^\cc)$. 
		Consider the functions $g_k:=f_k\circ \Phi_{z_k}\in H^1_c(W_{z_k})$.
		As all $W^{(z_k)}$ are convex, by Lemma~\ref{lem22convex} for any $\mu>0$ we have
		\begin{equation}\label{gkest}
			\int_{W_{z_k}\cap \partial W^{(z_k)}}|g_k|^2\dd S\le \dfrac{1}{\mu} 
			\int_{W_{z_k}\cap (W^{(z_k)})^\cc} |\nabla g_k|^2\dd x
			+\mu \int_{W_{z_k}\cap (W^{(z_k)})^\cc} |g_k|^2\dd x.
		\end{equation}
		Recall that for any $x\in W_{z_k}$ we have $\|D\Phi_{z_k}(x)-I\|<\eps$. Then for any $\xi\in\RR^n$ it holds that
		\begin{align*}
			\big|D\Phi_{z_k}(x)\xi\big|&\le |\xi| +\Big|\big(D\Phi_{z_k}(x)-I\big)\xi\Big|\le (1+\eps)|\xi|,\\
			\big|D\Phi_{z_k}(x)\xi\big|&\ge |\xi| -\Big|\big(D\Phi_{z_k}(x)-I\big)\xi\Big|\ge (1-\eps)|\xi|,
		\end{align*}
		i.e. one has the two-sided estimate
		\[
		(1-\eps)|\xi|\le\big|D\Phi_{z_k}(x)\xi\big|\le (1+\eps)|\xi|.
		\]
		It follows that $(1-\eps)^2I\le D\Phi_{z_k}(x)^t D\Phi_{z_k}(x)\le (1+\eps)^2 I$, hence, all eigenvalues of the matrix $D\Phi_{z_k}(x)^t D\Phi_{z_k}(x)$ are contained in $[(1-\eps)^2,(1+\eps)^2]$ due to the min-max principle.
		Therefore,
		\begin{gather*}
			\big| \det D\Phi_{z_k}(x)\big| =\sqrt{\det \big(D\Phi_{z_k}(x)^t D\Phi_{z_k}(x)\big)}\in\big[ (1-\eps)^n,(1+\eps)^n\big],\\
			\text{i.e. }(1-\eps)^n\le \big| \det D\Phi_{z_k}(x)\big|\le (1+\eps)^n.
		\end{gather*}
		It follows that
		\begin{gather}
			\label{ff4}
			\begin{aligned}
				\int_{U_{z_k}\cap U^\cc} |f_k|^2\dd x&=\int_{W_{z_k}\cap(W^{(z_k)})^\cc} |g_k|^2 |\det D \Phi_{z_k}|\dd x\\
				&\ge (1-\eps)^n \int_{W_{z_k}\cap(W^{(z_k)})^\cc} |g_k|^2 \dd x,
			\end{aligned}\\
			\begin{aligned}	
				\int_{W_{z_k}\cap(W^{(z_k)})^\cc} |\nabla g_k|^2\dd x&=\int_{W_{z_k}\cap(W^{(z_k)})^\cc} | D \Phi_{z_k}^t \cdot \big((\nabla f_k)\circ \Phi_{z_k}\big)|^2\dd x\\
				&\le (1+\eps)^2 \int_{W_{z_k}\cap(W^{(z_k)})^\cc} \big|(\nabla f_k) \circ \Phi_{z_k}\big|^2\dd x\\
				&\le \dfrac{(1+\eps)^2}{(1-\eps)^n} \int_{W_{z_k}\cap(W^{(z_k)})^\cc} \big| (\nabla f_k)\circ \Phi_{z_k}|^2 |\det D \Phi_{z_k}|\dd x\\
				&=\dfrac{(1+\eps)^2}{(1-\eps)^n}\int_{U_{z_k}\cap U^\cc} |\nabla f_k|^2\dd x.
			\end{aligned}
			\label{ff5}
		\end{gather}
		It remains to compare the boundary integrals of $f_k$ and $g_k$. Let $s\mapsto \varphi(s)$ be a local chart on $W_{z_k}\cap \partial W^{(z_k)}$, then $\Tilde\varphi:\ s\mapsto \Phi_{z_k}\circ \varphi (s)$ is a local chart on $U_{z_k}\cap \partial U$. For any $\xi\in\RR^{n-1}$ we have $\big|D \Tilde\varphi (s)\xi\big|=\big|D \Phi_{z_k} \big(\varphi(s)\big) D\varphi (s)\xi\big|$,
		therefore,
		\[
		(1-\eps ) \big|D\varphi (s)\xi\big|\le \big|D \Tilde\varphi (s)\xi\big|\le (1+\eps ) \big| D\varphi (s)\xi\big|,
		\]
		and it follows that
		\[
		(1-\eps)^2 D \varphi (s)^t D \varphi (s) \le D \Tilde\varphi (s)^t D \Tilde\varphi (s)\le(1+\eps)^2 D \varphi (s)^t D \varphi (s),
		\]
		which yields the respective two-sided estimate for the eigenvalues and results in the inequality for the determinants
		\[
		(1-\eps)^{n-1}\le \dfrac{\sqrt{\det \big(D \Tilde \varphi (s)^t D \Tilde\varphi (s)\big)}}{\sqrt{\det \big(D \varphi (s)^t D \varphi (s)\big)}}\le(1+\eps)^{n-1}.
		\]
		It follows that
		\begin{equation}
			\label{fgs}
			\int_{U_{z_k}\cap\partial U} |f_k|^2\dd S\le (1+\eps)^{n-1}\int_{W_{z_k}\cap\partial W^{(z_k)}} |g_k|^2\dd S.
		\end{equation}
		
		Now we obtain, for any $M>0$ and any $\mu>0$,
		\begin{align*}
			M\int_{\partial U}|f|^2\dd S&\stackrel{\eqref{ff3}}{=}M\sum_{k=1}^K\int_{U_{z_k}\cap \partial U}|f_k|^2\dd S\\
			&\stackrel{\eqref{fgs}}{\le} M(1+\eps)^{n-1}\sum_{k=1}^K\int_{W_{z_k}\cap\partial W^{(z_k)}}|g_k|^2\dd S\\
			&\stackrel{\eqref{gkest}}{\le}  \dfrac{M(1+\eps)^{n-1}}{\mu}\sum_{k=1}^K \int_{W_{z_k}\cap (W^{(z_k)})^\cc} |\nabla g_k|^2\dd x\\
			&\qquad  
			+\mu M(1+\eps)^{n-1}\sum_{k=1}^K \int_{W_{z_k}\cap (W^{(z_k)})^\cc} |g_k|^2\dd x\\
			&\stackrel{\eqref{ff4},\,\eqref{ff5}}{\le} \dfrac{M(1+\eps)^{n+1}}{\mu(1-\eps)^n}\sum_{k=1}^K \int_{U_{z_k}\cap U^\cc} |\nabla f_k|^2\dd x\\
			&\qquad
			+\dfrac{\mu M(1+\eps)^{n-1}}{(1-\eps)^n} \sum_{k=1}^K \int_{U_{z_k}\cap U^\cc} |f_k|^2\dd x\\
			&\stackrel{\eqref{ff1},\,\eqref{ff2}}\le \dfrac{M(1+\eps)^{n+1}}{\mu(1-\eps)^n} \Big(\int_{U^\cc} |\nabla f|^2\dd x +
			\int_{U^\cc} h |f|^2\dd x\Big)\\
			&\qquad +\dfrac{\mu M(1+\eps)^{n-1}}{(1-\eps)^n} \int_{U^\cc} |f|^2\dd x.
		\end{align*}
		By denoting $C:=\|h\|_\infty$ and taking
		\[
		\mu:=M\dfrac{(1+\eps)^{n+1}}{(1-\eps)^n}
		\]
		we arrive at
		\[
		\int_{U^\cc} |\nabla f|^2\dd x-M\int_{\partial U}|f|^2\dd S
		\ge -\Big[ \Big(\dfrac{1+\eps}{1-\eps}\Big)^{2n} M^2+C\Big]\int_{U^\cc} |f|^2\dd x.
		\]
		As $f\in H^1(U^\cc)$ was arbitrary, this results in
		\[
		\Lambda(U^\cc,M)\ge -\Big[ \Big(\dfrac{1+\eps}{1-\eps}\Big)^{2n} M^2+C\Big],
		\qquad \liminf_{M\to+\infty} \dfrac{\Lambda(U^\cc,M)}{M^2}\ge - \Big(\dfrac{1+\eps}{1-\eps}\Big)^{2n}.
		\]
		By taking arbitrarily small $\eps>0$ we conclude that 
		\[
		\liminf_{M\to+\infty} \dfrac{\Lambda(U^\cc,M)}{M^2}\ge -1,
		\]
		which is the required claim.
	\end{proof}
	
	Note that the result of Lemma~\ref{lemmpseudo} was known for the case of $C^1$ domains \cite{Pankrashkin_2020,lz}.
	
	\begin{theorem}\label{thmgen}
		Let $\Omega_+$ be locally $C^1$--convexifiable with compact boundary, then for any $m\in\RR$ and any $z\in\CC\setminus\RR$ one has
		\[
		\big\|R_{m,M}(z)-R^+_m(z)\oplus \zero\big\|\xrightarrow{M\to+\infty}0.
		\]
	\end{theorem}
	
	\begin{proof}
	Denote
	\[
	\rho(M):=\max\big\{-\Lambda(\Omega_-, M)-M^2,0\big\},
	\]
	then obviously $\rho:[0,\infty)\to[0,\infty)$ with $\Lambda(\Omega_-)\ge -M^2-\rho(M)$.
	By Lemma~\ref{lemmpseudo}, one has $\rho(M)=o(M^2)$ for $M\to+\infty$, and the conclusion follows by Theorem~\ref{thm19}.	
	\end{proof}

	We further observe that Theorem \ref{thmgen} also applies to the case where $\Omega_+$ is a finite disjoint union of convex domains, which Theorem \ref{thmconvex} does not cover in its current form. Additionally, a careful inspection of the preceding argument shows that, by specializing the proof of Lemma \ref{lemmpseudo} to the trivial case $\Phi_z= I$, one obtains convergence with rate $O(M^{-\frac{1}{2}})$. Building on this observation, we now introduce another class of domains:
	
	\begin{definition}\label{deflocconv}
		An open set $U\subset\RR^n$ will be called \emph{locally convex-congruent} if for every $z\in \partial U$ there
		are an open neighborhood $U_z$ of $z$ in $\RR^n$ and an open set $W^{(z)}\subset \RR^n$
		that is either convex or has compact $C^{1,1}$ boundary, such that  $U\cap U_z = W^{(z)}\cap U_z$.
	\end{definition}
	
	This corresponds in the convex case exactly to Definition \ref{deff} with $\Phi_z= I$ and $U_z=W_z$ everywhere, which implies in combination with Lemma \ref{lemC1convexifiable} that every locally convex-congruent domain $\Omega_+$ is, in particular, locally $C^1$-convexifiable. Consequently, Theorem \ref{thmgen} already applies but only with the convergence rate $o(1)$. The next theorem shows that this rate can be improved to $O(M^{-\frac{1}{2}})$.
	
	\begin{theorem}\label{thmlocconvex}Let $\Omega_+\subset\RR^n$ be locally convex-congruent with compact boundary. Then there is a constant $C_0\geq 0$, independent of $M$, such that
		\[ 
		\Lambda(\Omega_-,M)\ge -M^2 -C_0M\quad \text{for all }M>1.
		\]
		Consequently, for every $m\in\RR$ and $z\in\CC\setminus\RR$ one has 
		\[
		\big\|R_{m,M}(z)-R^+_m(z)\oplus \zero\big\|=O\Big(\dfrac{1}{\sqrt{M}}\Big) \text{ as } M\to +\infty.
		\]
	\end{theorem}
	\begin{proof}Keeping the same notations as in the proof of Lemma \ref{lemmpseudo} with $\Phi_{z_k}= I$, we may take  $U_{z_k}=W_{z_k}$ and $\eps=0$ throughout. Hence, $g_k:=f_k\circ \Phi_{z_k}= f_k$, the Jacobian factors in the inequalities \eqref{ff4}--\eqref{ff5} are exactly $1$, and $U_{z_k}\cap \partial W^{({z_k})}=U_{z_k}\cap \partial\Omega_+$. Consequently, applying Lemma~\ref{lem-kovp} or Lemma~\ref{lem22convex} to each $W^{({z_k})}$ (without coordinate change), the bounds \eqref{gkest} and \eqref{fgs} become for some fixed $c_k \geq 0$ and for every $\mu>0$, 
		\begin{equation*}
			\int_{U_{z_k}\cap \partial\Omega_+}|f_k|^2\dd S\le \dfrac{1}{\mu} 
			\int_{U_{z_k}\cap \Omega_-} |\nabla f_k|^2\dd x
			+(\mu+c_k) \int_{U_{z_k}\cap \Omega_-} |f_k|^2\dd x.
		\end{equation*}
		Summing over $k$ as in the proof of Lemma \ref{lemmpseudo}, using the decomposition \eqref{ff1}--\eqref{ff3} there and choosing $\mu=M$, gives
		\begin{multline*}
			M\int_{ \partial\Omega_-}|f|^2\dd S\le  
			\int_{ \Omega_-} |\nabla f|^2\dd x + \int_{\Omega_-} \bigl(h+M\max \{c_1,\dots,c_K \}\bigr) |f|^2\dd x\\
			+M^2 \int_{\Omega_-} |f|^2\dd x,\qquad h:=\sum_{k=0}^{K} |\nabla \psi_k|^2.
		\end{multline*}
		Taking $C_0= \|h\|_\infty+\max \{c_1,\dots,c_K \} $ (which depends only on the fixed partition of unity) yields that  $\Lambda(\Omega_-,M)\ge -M^2 -MC_0$  for every $M>1$. Therefore, the assumption \eqref{assump2} is satisfied with $\rho(M)=O(M)$, and we conclude the proof by using Theorem~\ref{thm19}.
	\end{proof}

	\section{Generalizations and possible extensions}\label{sec6}
	 The observations of the recent work~\cite{duran} allow us to extend the above results to other boundary conditions.
	
	First, for $m,\mu,\eps\in\RR$ consider the following operator $\DD_{m,\mu,\eps}$ in $L^2(\RR^n,\CC^N)$:
	\[
	\begin{aligned}
		\DD_{m,\mu,\eps} f&:=D_m f \text{ in }\RR^n\setminus\Sigma,\\
		\dom \DD_{m,\mu,\eps}&:=\Big\{
		f\in H^\half_\alpha(\RR^n\setminus\Sigma,\CC^N):\\
		&\quad \qquad \rmi (\alpha\cdot\nu)(\gamma_+f-\gamma_-f)+(\eps I + \mu\beta)\dfrac{\gamma_+f + \gamma_-f}{2} =0
		\Big\}.	
	\end{aligned}
	\]
	In particular, $\DD_{m,\mu}=\DD_{m,\mu,0}$. Analogously to $\DD_{m,\mu}$, one often uses the formal writing
	$\DD_{m,\mu,\eps}=\DD_m+(\mu\beta+\eps I)\delta_\Sigma$, and $\mu$ and $\eps$ are interpreted as the strengths of Lorentz
	and electrostatic $\delta$-interactions supported on $\Sigma$, see e.g. \cite{BoundaryTripleWeylfunctions}.
	
	Due to \cite[Corol.~3.2]{duran}, if $\eps^2-\mu^2=\eps_0^2-\mu^2_0\ne 0$, then the self-adjointness of $D_{m,\mu_0,\eps_0}$ is equivalent to the self-adjointness of $D_{m,\mu,\eps}$. As a direct consequence, we obtain:
	\begin{theorem}
		For any $m,\eps,\mu\in\RR$ with $\eps^2-\mu^2<0$ the operator $\DD_{m,\mu,\eps}$ is self-adjoint.
	\end{theorem}	
	
	\begin{proof}
		Let $\mu_0\in \RR$ with $-\mu_0^2=\eps^2-\mu^2$. Then $\DD_{m,\mu_0,0}\equiv \DD_{m,\mu_0}$ is self-adjoint by Theorem~\ref{thm-sa}, and the preceding equivalence yields the self-adjointness of $\DD_{m,\mu,\eps}$.
	\end{proof}
	
	Further, let $\Omega\subset\RR^n$ be an open set with compact Lipschitz boundary, and for
	\[
	\mu,\eps\in\RR,\qquad \mu^2-\eps^2=4,
	\]
	consider the following operator $A^\Omega_{m,\mu,\eps}$ in $L^2(\Omega,\CC^N)$:
	\begin{gather*}
		A^\Omega_{m,\mu,\eps}:\ f\mapsto D_m f,\\
		\dom A^\Omega_{m,\mu,\eps}:=\big\{f\in H^\half_\alpha(\Omega,\CC^N):\ f= -\dfrac{\rmi}{2} (\eps I +\mu \beta) (\alpha\cdot\nu) f \text{ on }\partial\Omega\big\},
	\end{gather*}
	which are known as Dirac operators with generalized MIT bag boundary conditions, see e.g. \cite{amas}. Note that the MIT bag operator $A^\Omega_m$
	corresponds to $\mu=2$ and $\eps=0$. As noticed in \cite[Lem.~4.1]{duran}, the self-adjointness of $A^\Omega_m$
	is equivalent to the self-adjointness of all $A^\Omega_{m,\mu,\eps}$, and by applying Lemma~\ref{lem10} we arrive at:
	\begin{theorem}\label{thm62}
		Let $\Omega\subset\RR^n$ be an open set with compact Lipschitz boundary and $m\in\RR$, then all generalized MIT bag operators $A^\Omega_{m,\mu,\eps}$ with $\mu^2-\eps^2=4$, $\mu,\eps\in\RR$,  are self-adjoint.
	\end{theorem}

	Finally, an application of \cite[Lem.~4.4]{duran} allows us to transfer the infinite mass interpretation of $A^\Omega_m$ to a similar version for $A^\Omega_{m,\mu,\eps}$:
	\begin{theorem}\label{thm63}
		Let $\Omega\subset\RR^n$ be a locally $C^1$-convexifiable open set with compact boundary. Then for any $m\in\RR$, any generalized MIT bag operator $A^\Omega_{m,\mu,\eps}$ and any $z\in\CC\setminus\RR$ one has
		\[
		(B_{m,M,\mu,\eps}-z)^{-1}\xrightarrow{M\to+\infty}(A^\Omega_{m,\mu,\eps}-z)^{-1}\oplus \zero
		\]
		in the norm, where
		\[
		B_{m,M,\mu,\eps}:=\DD_0 +m\one_{\Omega}\beta +\half M\one_{\Omega^\cc}(\mu\beta-\eps I).
		\]
	\end{theorem}
	Improved convergence rates as in Theorem~\ref{thmlocconvex} are directly transferred as well.

    \begin{remark}
        The proof of the resolvent convergence in Lemma \ref{lem17conv} strongly depends on the assumption \eqref{assump}.
        When taking a closer look at the usage of this assumption in estimate (\ref{Qestimate}), one may notice that everything still works under the slightly weaker assumption that there exists a function $\Tilde \rho: [0,\infty) \to [0,\infty)$ with
        \[
        \lim_{M \to + \infty} \frac{\Tilde \rho(M)}{M^2} =0
        \]
        such that for any $f \in H^1(\Omega_-) $ and $M>0$ it holds that
        \[
        \int_{\Omega_-} | \nabla f|^2 \dd x - M \int_{\Sigma} |\cP_+ \gamma f|^2 \dd S\geq - \big(M^2 + \Tilde \rho(M)\big) \int_{\Omega_-} |f|^2\dd x.
        \]
        While this may be considered as a small technical improvement, an easy analysis shows that it still does not allow
        for the inclusion of general Lipschitz domains. For example, the above estimate fails
        if $\Omega_- \subseteq \RR^2$ is a polygonal domain with sufficiently acute angles. In fact, we were unable to use this weaker assumption in order to extend the class of domains for which the infinite mass limit can be established.        
    \end{remark}

\begin{remark}
    Let us also discuss what happens if assumption (\ref{assump}) does not hold. While the preceding argument does not prove resolvent convergence in this case, one may establish a stability result showing that Theorem \ref{thm19} does not break down abruptly as $\Lambda(\Omega_-,M)$ falls below $-M^2 -o(M^2)$. Recall that (Lemma \ref{lem16est}) without the hypothesis \eqref{assump} there are $C_U, C_R^+>0$ (depending only on $\Omega_+$, $m$, and $z$) such that 
	\begin{align}\label{EQQ1}
		\big\| \cP_+ \gamma R_{m,M}(z)\big\|&\le C_U\big\|\cP_- \gamma R_{m,M}(z)\big\|+C_R^+ \quad \text{for all $M>0$},
	\end{align}
	and there are $M_0>0$ and constants $C_E,C_R>0$ (independent of $\Omega_+$ and $m$) with
	\begin{align*}
		\big\|R_M^-(z)\big\| \le C_R M^{-1},\quad \big\|E^-_M(z)\big\|\le C_E M^{-\half} \quad \text{for all $M>M_0$}.
	\end{align*}
	For $M>0$ we set 
	\begin{align*}
		C_M:=\max \left\{ 1, -\frac{\Lambda(\Omega_-,M)}{M^2}\right\}, \quad C_\infty= \limsup_{M\to \infty} C_M.
	\end{align*}
	Note that $C_\infty \in[1,\infty)$ since for any Lipschitz domain $U\subset\RR^n$ there is $1\leq c_{U}\leq \infty$ such that $-c_U M^2\leq \Lambda(U,M)\leq -M^2$ for $M>0$ large, see \cite[Thm.~2.1]{GS} for the upper bound and \cite[Lemma 2.7]{kobp} for the lower bound.
	\end{remark}
	
	\begin{proposition}\label{thmstability}Fix $m\in\RR$ and $z\in\CC\setminus\RR$, and set 
		\[
		\kappa:=\frac{1}{4C_E^2 C_U^2}.
		\]
		If $ C_\infty<1+\kappa$ then  
		\[
		\limsup_{M\to\infty}\big\|R_{m,M}(z)-R^+_m(z)\oplus \zero\big\|\leq \Upsilon_m(C_\infty),
		\]
		where $\Upsilon_m:[1,1+\kappa)\to[0,\infty)$ is continuous and increasing with $\Upsilon_m(1)=0$. In particular, the resolvent convergence is recovered continuously as $C_\infty \to 1^+$. 
	\end{proposition}
	\begin{proof} We only give the proof for $m=0$, the case $m\neq 0$ follows straightforwardly adapting the arguments used in the proof Lemma \ref{lem18rate}.
		
		Fix $g\in L^2(\RR^n,\CC^N)$ and set $f:=R_{0,M}(z)g\in H^1(\RR^n,\CC^N)$. By \eqref{Qestimate} we have 
		\begin{align*}
			\|A_{0,M}f\|^2_{L^2(\RR^n,\CC^N)}&\geq (\Lambda(\Omega_-,M) +M^2)\|f\|^2_{L^2(\Omega_-,\CC^N)}+M\|\cP_- \gamma f\|^2_{L^2(\Sigma,\CC^N)}\\
			&\geq -(C_M-1) M^2\|f\|^2_{L^2(\Omega_-,\CC^N)}+M\|\cP_- \gamma f\|^2_{L^2(\Sigma,\CC^N)},
		\end{align*}
		thus    
		\begin{align}\label{Star}
			M\|\cP_- \gamma f\|^2_{L^2(\Sigma,\CC^N)} \leq \|A_{0,M}f\|^2_{L^2(\RR^n,\CC^N)} +(C_M-1) M^2\|f\|^2_{L^2(\Omega_-,\CC^N)}.
		\end{align}
		Set	$X\equiv X(M):= \|\cP_- \gamma R_{0,M}(z)\|$  and $Y\equiv Y(M):=\|\cP_+ \gamma R_{0,M}(z)\|$, then by \eqref{exterior resolvent perturbed op} we have  
		\begin{align*}
			\|f\|_{L^2(\Omega_-,\CC^N)}&=	\big\|r_-R_{0,M}(z)g\big\|_{L^2(\Omega_-,\CC^N)}\\
			&\leq \big\|E_{M}^{-}(z)\big\|\cdot \big\|\cP_+\gamma R_{0,M}(z)g\big\|_{L^2(\Sigma,\CC^N)} + \big\|R_{M}^{-}(z)g\big\|\\
			&\le \big( C_EM^{-\frac{1}{2}}Y + C_RM^{-1}\big) 	\|g\|_{L^2(\RR^n,\CC^N)},
		\end{align*}
		and this implies
		\begin{align*}
			\|f\|^2_{L^2(\Omega_-,\CC^N)}\le 2\big( C_E^2 M^{-1}Y^2 + C_R^2 M^{-2}\big) 	\|g\|^2_{L^2(\RR^n,\CC^N)}.
		\end{align*}
		Substituting into \eqref{Star}, using $\|A_{0,M}f\|= \|g+ zf\|\leq \tau \|g\|$ with $\tau:=1+ \dfrac{|z|}{|\Im(z)|}$, and taking the supremum over $\|g\|\leq 1$ gives 
		\begin{align*}
			X^2 \leq \frac{\tau^2}{M} +2C^2_E(C_M-1)Y^2 + \frac{2C_R^2(C_M-1)}{M},
		\end{align*}
		and inserting the inequality  $Y^2\leq 2 C_U^2 X^2 +2(C^+_R)^2$, yields
		\begin{align*}
			X^2(1-4C^2_E C_U^2(C_M-1) ) \leq   \frac{\tau^2+ 2C_R^2(C_M-1)}{M}  +4C^2_E(C^+_R)^2(C_M-1).
		\end{align*}
		If $C_\infty - 1<\kappa\equiv (4C^2_E C_U^2)^{-1}$, choose $\eta>0$ with $C_\infty+\eta - 1<\kappa$; for $M$ large, $C_M< C_\infty +\eta$, so the bracket on the left is bounded below by a fixed positive number, and letting $M\to\infty$ then $\eta\to 0$, we get
		\begin{align*}
			\limsup_{M\to\infty}X(M)^2\leq l^2_\infty:=  \frac{4C^2_E(C^+_R)^2(C_\infty-1)}{1-(C_\infty-1)/\kappa}.
		\end{align*}
		By \eqref{EQQ1} again, we have $\limsup_{M\to\infty} Y(M)\leq C_U l_\infty + C_R^+=: L_\infty<\infty$. Hence,  as in the proof of Lemma \ref{lem17conv}, we have the resolvent decomposition $R_{0,M}(z)-R^+_0(z)\oplus \zero=J_1+J_2+J_3$, with 
		\begin{align*}
			\| J_1\|&:= \big\|e_+E^+_0(z)\cP_- \gamma R_{0,M}(z)\big\| \leq \big\| E_0^+(z)\big\|X(M)\\
			&\limsup_{M\to+\infty}\| J_1\|\le \big\| E_0^+(z)\big\| l_\infty,\\
			\|J_2\|&:=\|e_-E^-_M(z)\cP_+ \gamma R_{0,M}(z)\|\leq \frac{C_E Y(M)}{\sqrt{M}},\\
			&\limsup_{M\to+\infty}\| J_2\|\leq \limsup_{M\to\infty}\frac{C_E L_\infty}{\sqrt{M}}= 0,\\
			\|J_3\|&:=\|e_-R^-_M(z)r_-\|\leq \frac{C_R}{M}\longrightarrow 0.
		\end{align*}
		Set $\Upsilon_0(C_\infty):=\big\| E_0^+(z)\big\| l_\infty$. By definition of $l_\infty$, it is clear that $\Upsilon_0(C_\infty)$ is continuous and increasing with $\Upsilon_0(C_\infty)\xrightarrow{C_\infty \to 1^+}0$. Therefore,
		\[
		\limsup_{M\to\infty}\big\|R_{m,M}(z)-R^+_m(z)\oplus \zero\big\|\leq \Upsilon_0(C_\infty),
		\]
		as claimed.
	\end{proof}
	
	\begin{remark}
		The proof is expected to follow the same lines in the case of domains with non-compact boundaries, as soon as suitable mapping properties of the trace maps and layer operators are established. However, stating precise assumptions guaranteeing the validity of all necessary ingredients in the case of 	boundaries that
		are simultaneously non-smooth and non-compact is likely to be very demanding, so we prefer to discuss these aspects separately somewhere else.
	\end{remark}

	\begin{remark}
		Note that the operator $A^\Omega_{m,M}$ is of the form $A+ M B$ with fixed self-adjoint $A$ and $B$, so it
		formally belongs to the framework of large coupling convergence. Although there are many works dealing with the
		case of semi-bounded $A$ and $B$, see e.g. the review in \cite{benamor}, the theory for non-sign-definite $A$ and/or $B$ is surprisingly incomplete as recently discussed in \cite{koke}. While some initial results
		of general $A$ and $B$ were obtained in \cite{koke}, none of the initial assumptions applies to our situation, and
 it remains unclear how an analog of the MIT bag boundary condition may arise in the abstract setting. This is an interesting direction for future work.
	\end{remark}

	\section*{Acknowledgments}
	N.~K\"orner and K.~Pankrashkin were partially supported by the Deutsche For\-schungs\-gemeinschaft (German Research Foundation), project 491606144.
	
	This article is based upon work from COST Action 24122 mSPACE, supported by COST (European Cooperation in Science and Technology), \url{http://cost.eu}.
	
	A large part of the work was prepared during a visit of D.~Machado to the Carl von Ossietzky Universit\"at Oldenburg in May--June 2026, and he would like to thank the Institut f\"ur Mathematik for the warm hospitality.

\end{document}